\documentclass[a4paper, 11pt]{amsart}
\usepackage[utf8]{inputenc}
\usepackage[T1]{fontenc}
\usepackage[english]{babel}
\usepackage{graphicx}
\usepackage{stmaryrd}
\usepackage{tikz}
\usepackage{tikz-cd}
\usepackage[all]{xy}
\usepackage{comment}
\usepackage[a4paper]{geometry}
\usepackage{amsmath,amsfonts,amssymb,amsthm,epsfig,epstopdf,url,array}
\usepackage{rotating}
\usepackage{cleveref}
\usepackage{nameref}
\usepackage{enumitem}
\usepackage{comment}
\Crefname{paragraph}{Section}{Sections}
\usepackage{fancyhdr}

\newcommand{\ensemblenombre}[1]{\mathbb{#1}}
\newcommand{\N}{\ensemblenombre{N}}

\newcommand{\R}{} 
\renewcommand{\R}{\ensemblenombre{R}}
\newcommand{\C}{\ensemblenombre{C}}

\newcommand{\norme}[1]{\left\lVert#1\right\rVert}

\theoremstyle{plain} 
\newtheorem{prop}{Proposition}[section] 
\newtheorem{theo}[prop]{Theorem}
\newtheorem{defprop}[prop]{Definition-Proposition}
\newtheorem{lem}[prop]{Lemma}
\newtheorem{cor}[prop]{Corollary}
\theoremstyle{definition}

\newtheorem{rmk}[prop]{Remark}
\newtheorem{app}[prop]{Application}

\newtheorem{ass}[prop]{Assumption}
\newtheorem{Exam}[prop]{Example}
\newtheorem{op}[prop]{Open problem}

\makeatletter
\let\original@addcontentsline\addcontentsline
\newcommand{\dummy@addcontentsline}[3]{}
\newcommand{\DeactivateToc}{\let\addcontentsline\dummy@addcontentsline}
\newcommand{\ActivateToc}{\let\addcontentsline\original@addcontentsline}
\makeatother

\begin{document}
\title{Local controllability of reaction-diffusion systems around nonnegative stationary states}
\author{Kévin Le Balc'h}
\address{Kévin Le Balc'h, Univ Rennes, ENS Rennes, CNRS, IRMAR - UMR 6625, F-35000 Rennes, France}
\email{kevin.lebalch@ens-rennes.fr}
\maketitle
\begin{center}
\today
\end{center}
\begin{abstract}
We consider a $n \times n$ nonlinear reaction-diffusion system posed on a smooth bounded domain $\Omega$ of $\R^N$. This system models reversible chemical reactions. We act on the system through $m$ controls ($1 \leq m < n$), localized in some arbitrary nonempty open subset $\omega$ of the domain $\Omega$. We prove the local exact controllability to nonnegative (constant) stationary states in any time $T >0$. A specificity of this control system is the existence of some invariant quantities in the nonlinear dynamics that prevents controllability from happening in the whole space $L^{\infty}(\Omega)^n$. The proof relies on several ingredients. First, an adequate affine change of variables transforms the system into a cascade system with second order coupling terms. Secondly, we establish a new null-controllability result for the linearized system thanks to a spectral inequality for finite sums of eigenfunctions of the Neumann Laplacian operator, due to David Jerison, Gilles Lebeau and Luc Robbiano and precise observability inequalities for a family of finite dimensional systems. Thirdly, the source term method, introduced by Yuning Liu, Takéo Takahashi and Marius Tucsnak, is revisited in a $L^{\infty}$-context. Finally, an appropriate inverse mapping theorem enables to go back to the nonlinear reaction-diffusion system.
\end{abstract}

\small
\tableofcontents
\normalsize
\section{Introduction}
\subsection{Free system}
\indent We consider the following reversible chemical reaction: 
\begin{equation}\label{systchim}
\alpha_1 A_1 + \dots + \alpha_n A_n \rightleftharpoons \beta_1 A_1 + \dots + \beta_n A_n,
\end{equation}
where $n \in \N^{*}$, $A_1,\dots,A_n$ denote $n$ chemical species and $(\alpha_1, \dots, \alpha_n)$, $(\beta_1, \dots, \beta_n)$ belongs to $(\N)^n$ and are such that for every $1 \leq i \leq n$, $\alpha_i \neq \beta_i$.\\
\indent For $1 \leq i \leq n$, let $u_i(t,.) : \Omega \rightarrow \R$ be the concentration of the chemical component $A_i$ at time $t$. By using the law of mass action and Fick's law, $(u_i)_{1 \leq i \leq n}$ satisfies the following reaction-diffusion system:
\begin{align}\label{systNLF}
\left\{
\begin{array}{l l}
\forall 1 \leq i \leq n,\\
\partial_t u_i - \underbrace{d_i \Delta u_{i}}_{\text{diffusion}} = \underbrace{(\beta_i-\alpha_i)\left(\prod\limits_{k=1}^n u_k^{\alpha_k} - \prod\limits_{k=1}^n u_k^{\beta_k}\right)}_{\text{reaction}} &\mathrm{in}\ (0,T)\times\Omega,\\
\frac{\partial u_i}{\partial \nu} = 0 &\mathrm{on}\ (0,T)\times\partial\Omega,\\
u_i(0,.)=u_{i,0} &\mathrm{in}\  \Omega,
\end{array}
\right.
\end{align}
where $T \in (0,+\infty)$, $\Omega$ is a bounded, connected, open subset of $\R^N$ (with $N \geq 1$) of class $C^2$, $\nu$ is the outer unit normal vector to $\partial\Omega$ and for every $1 \leq i \leq n$, $d_i \in (0,+\infty)$ is the diffusion coefficient of the chemical species $A_i$.\\
\indent In general, global existence of classical solutions (in the sense of \cite[Definition (1.5)]{P}) or weak solutions (in the sense of \cite[Definition (5.12)]{P}) for \eqref{systNLF} is an open problem.
\begin{itemize}
\item For particular semilinearities with a so-called triangular structure (see \cite[Section 3.3]{P}), classical solutions exist in the time interval $[0,+\infty)$ and are unique. For example, take $n=2$, $\alpha_1 \geq 1$,  $\beta_2 = 1$, $\alpha_2 = \beta_1 = 0$ and apply \cite[Theorem 3.1]{P}. 
\item For at most quadratic nonlinearities, global existence of weak solutions holds (see \cite[Theorem 5.12]{P}). For instance, take $n=4$, $\alpha_1 = \alpha_3 = \beta_2 = \beta_4 = 1$, $\alpha_2 = \alpha_4 = \beta_1= \beta_3 = 0$. For spatial dimension $N \geq 3$, the recent works \cite{CGV} and \cite{SoP} (inspired by the previous works \cite{Kanel1} and \cite{Kanel2}) prove that the solutions are bounded for bounded initial data, which ensure global existence of classical solutions.
\item Without a priori $L^1$-bound on the nonlinearities, a challenging problem is to understand whether global solutions exist. For example, take $n=2$, $\alpha_1 = \beta_2 = 2$, $\beta_1 =  \alpha_2 = 3$ (see \cite[Problem 1]{P}).
\end{itemize}
Let us also mention that global existence of renormalized solutions holds in all cases for \eqref{systNLF} (see \cite{Fisch}).
\subsection{Control system and open question}
\label{SectionControlOpQuestion}
We assume that one can act on the system through controls localized on a nonempty open subset $\omega$ of $\Omega$. From a chemical viewpoint, it means that one can add or remove chemical species at a specific location of the domain $\Omega$.  More precisely, let 
\begin{equation}
\label{DefJ}
J \subset \{1, \dots, n\}\ \textrm{and}\  := \# J < n \ \text{be the number of controls.}
\end{equation}
We consider the control system:
\begin{align}
\label{systNLC}
\tag{NL-U}
\left\{
\begin{array}{l l}
\forall 1 \leq i \leq n,\\ 
\partial_t u_i - d_i \Delta u_{i} =\\ \qquad(\beta_i-\alpha_i)\left(\prod\limits_{k=1}^n u_k^{\alpha_k} - \prod\limits_{k=1}^n u_k^{\beta_k}\right) + h_i 1_{\omega} 1_{i \in J} &\mathrm{in}\ (0,T)\times\Omega,\\
\frac{\partial u_i}{\partial \nu} = 0 &\mathrm{on}\ (0,T)\times\partial\Omega,\\
u_i(0,.)=u_{i,0} &\mathrm{in}\  \Omega,
\end{array}
\right.
\end{align}
where $1_{i \in J} := 1\ \text{if}\ i \in J\ \text{and}\ 0\ \text{if}\ i \notin J$. Here, $(u_i(t,.))_{1 \leq i \leq n} : \Omega \rightarrow \R^n$ is the \textit{state} to be controlled, $(h_i(t,.))_{i \in J}  : \Omega \rightarrow \R^m$ is the \textit{control input} supported in $\omega$.\\
\indent Let $(u_i^*)_{1 \leq i \leq n}$ be a nonnegative stationary state of \eqref{systNLF} i.e.
\begin{equation}\label{Stat}
\forall 1 \leq i \leq n,\ u_i^* \in [0,+\infty)\ \text{and}\  \prod\limits_{k=1}^n {u_k^{*}}^{\alpha_k} = \prod\limits_{k=1}^n {u_k^{*}}^{\beta_k}.
\end{equation}
Note that the nonnegative stationary solutions of \eqref{systNLF} do not depend on the space variable  (see \Cref{solutionspositivesconstantes} in \Cref{StatioStates}). Thus, it is not restrictive to assume that $(u_i^{*})_{1 \leq i \leq n} \in [0,+\infty)^n$.\\
\indent The question we ask is the following one: For a given initial condition $(u_{i,0})_{1 \leq i \leq n}$, does there exist $(h_i)_{i \in J}$ such that the solution $(u_i)_{1 \leq i \leq n}$ of \eqref{systNLC} satisfies 
\begin{equation} 
\forall i \in \{1,\dots,n\},\ u_i(T,.) = u_i^*?
\label{conditionfinale}
\end{equation}
\indent Under appropriate assumptions (see \Cref{Assdkneqdl} and \Cref{asscouplordre0} below), we prove the controllability of \eqref{systNLC}, in an appropriate subspace of $L^{\infty}(\Omega)^n$, locally around $(u_i^{*})_{1 \leq i \leq n}$, with controls in $L^{\infty}((0,T)\times\Omega)^m$ (see \Cref{mainresult1} below). \\
\indent By an adequate affine transformation, the proof relies on the study of the null-controllability of an equivalent cascade system with second order coupling terms (see \Cref{sectionChangeVar} below).

\subsection{Bibliographical comments}

In this section, we recall some known results about the null-controllability of linear and semilinear parabolic systems with Neumann boundary conditions. We investigate the case of one control, i.e., $m=1$ in this section to simplify. We introduce the notation $$Q_T := (0,T)\times\Omega.$$
\subsubsection{Linear results} The null-controllability of the heat equation was proved independently by Gilles Lebeau, Luc Robbiano in 1995 (see \cite{LR}, \cite{JeLe} and the survey \cite{LLR}) and by Andrei Fursikov, Oleg Imanuvilov in 1996 (see \cite{FI} and \cite{FCG}).
\begin{theo}\cite[Chapter I, Theorem 2.1]{FI}\\
For every $z_0 \in L^2(\Omega)$, there exists $h \in L^2(Q_T)$ such that the solution $z$ of 
\[
\left\{
\begin{array}{l l}
\partial_t z-  \Delta z =  h 1_{\omega} &\mathrm{in}\ (0,T)\times\Omega,\\
\frac{\partial z}{\partial \nu} = 0 &\mathrm{on}\ (0,T)\times\partial\Omega,\\
 z(0,.)=z_0  &\mathrm{in}\ \Omega,
\end{array}
\right.
\]
satisfies $z(T,.)=0$.
\end{theo}
In the work \cite{FCGBGP}, Enrique Fernández-Cara, Manuel González-Burgos, Sergio Guerrero and Jean-Pierre Puel prove the same null-controllability result for more general parabolic operators, i.e., $\partial_t z  - \Delta z  + B(t,x) . \nabla z + a(t,x) z $ with $a \in L^{\infty}(Q_T)$, $B \in L^{\infty}(Q_T)^n$ and linear Robin conditions, i.e., $\frac{\partial z}{\partial \nu} + \beta(t,x) z = 0$ on $(0,T)\times\partial\Omega$ with $\beta \in L^{\infty}((0,T)\times\partial\Omega;\R^{+})$.\\
\indent Then, the null-controllability of coupled linear parabolic systems has been a challenging issue. Let us now focus on a \textit{cascade} system with \textit{coupling terms of zero order}. The following result comes from an easy adaptation of Manuel González-Burgos and Luz de Teresa's proof in the case of Dirichlet boundary conditions to Neumann boundary conditions.
\begin{theo}\label{TheoCascade}\cite[Theorem 1.1]{GBT}\\
Let $(d_i)_{1 \leq i \leq n} \in (0,+\infty)^{n}$, $(a_{i,j})_{1 \leq i ,j \leq n} \in \R^{n \times n}$ and assume that $a_{i+1,i} \neq 0$ for every $1 \leq i \leq n$. Then, for every $(z_{i,0})_{1 \leq i \leq n} \in L^2(\Omega)^n$, there exists $h \in L^2(Q_T)$ such that the solution $(z_i)_{1 \leq i \leq n}$ of 
\[
\left\{
\begin{array}{l l l}
\partial_t z_1-  d_1 \Delta z_1 = \sum\limits_{k=1}^n a_{1,k} z_k +  h 1_{\omega} &\mathrm{in}\ (0,T)\times\Omega,&\\
\partial_t z_i-  d_i \Delta z_i = a_{i,i-1} z_{i-1} &\mathrm{in}\ (0,T)\times\Omega,&2 \leq i \leq n,\\
\frac{\partial z_i}{\partial \nu} = 0 &\mathrm{on}\ (0,T)\times\partial\Omega,& 1 \leq i \leq n,\\
 z_i(0,.)=z_{i,0}  &\mathrm{in}\ \Omega, & 1 \leq i \leq n,
\end{array}
\right.
\]
satisfies $z_i(T,.)=0$ for every $1 \leq i \leq n$.
\end{theo}
Roughly speaking, the component $z_1$ is controlled by the control input $h$, the component $z_2$ is controlled by $z_1$ thanks to the coupling term $a_{2,1}z_1$, \dots, the component $z_n$ is controlled by $z_{n-1}$ thanks to the coupling term $a_{n,n-1} z_{n-1}$.\\
\indent The following result for a $2 \times 2$ linear parabolic system with a “cross-diffusion” term is due to Sergio Guerrero. We introduce the function space $$L_{\sigma}^{2}(\Omega):= \left\{ z \in L^2(\Omega)\ ;\ \int_{\Omega} z = 0\right\}.$$
\begin{theo}\label{Guerrero}\cite[Theorem 1]{G}\\
Let $(d_1,d_2) \in (0,+\infty)^2$, $a,b \in \R$. Then, for every $(z_{1,0},z_{2,0}) \in L^2(\Omega) \times L_{\sigma}^{2}(\Omega)$, there exists $h \in L^2(Q_T)$ such that the solution $(z_1,z_2)$ of 
\[
\left\{
\begin{array}{l l}
\partial_t z_1-  d_1 \Delta z_1 = a z_1 + b z_2 +  h 1_{\omega} &\mathrm{in}\ (0,T)\times\Omega,\\
\partial_t z_2-  d_2 \Delta z_2 = \Delta z_1 &\mathrm{in}\ (0,T)\times\Omega,\\
\frac{\partial z_1}{\partial \nu} = \frac{\partial z_2}{\partial \nu} = 0 &\mathrm{on}\ (0,T)\times\partial\Omega,\\
 (z_1,z_2)(0,.)=(z_{1,0},z_{2,0})  &\mathrm{in}\ \Omega,
\end{array}
\right.
\]
satisfies $(z_1,z_2)(T,.)=0$.
\end{theo}
The main difference with \Cref{TheoCascade} is that now, the \textit{coupling term is of second order $\Delta z_1$}. The condition $\int_{\Omega} z_{2,0} = 0$ is necessary because by integrating with respect to the space variable the second equation of the system, we get 
$$ \frac{d}{dt} \int_{\Omega} z_2(t,.) = 0.$$
In particular if $z_2(T,.) = 0$, then we need $\int_{\Omega} z_{2,0} = 0$.\\
\indent In view of \Cref{TheoCascade}  and \Cref{Guerrero}, a natural question is: are cascade cross-diffusion systems (of arbitrary size $n \geq 2$) null-controllable? A byproduct of this article is a positive answer to this question.
\begin{theo}\label{TheoCascadeDeux}
Let $(d_i)_{1 \leq i \leq n} \in (0,+\infty)^{n}$, $(a_{i,j})_{1 \leq i ,j \leq n} \in \R^{n \times n}$ and assume that $a_{2,1} \neq 0$. Then, for every $(z_{i,0})_{1 \leq i \leq n} \in L^2(\Omega)\times L_{\sigma}^2(\Omega)^{n-1}$, there exists $h \in L^2(Q_T)$ such that the solution $(z_i)_{1 \leq i \leq n}$ of 
\[
\left\{
\begin{array}{l l l}
\partial_t z_1-  d_1 \Delta z_1 = \sum\limits_{k=1}^n a_{1,k} z_k +  h 1_{\omega} &\mathrm{in}\ (0,T)\times\Omega,&\\
\partial_t z_2-  d_2 \Delta z_2 =   a_{2,1} z_1 &\mathrm{in}\ (0,T)\times\Omega,&\\
\partial_t z_i-  d_i \Delta z_i =  \Delta z_{i-1} &\mathrm{in}\ (0,T)\times\Omega,& 3 \leq i \leq n,\\
\frac{\partial z_i}{\partial \nu} = 0 &\mathrm{on}\ (0,T)\times\partial\Omega,& 1 \leq i \leq n,\\
 z_i(0,.)=z_{i,0}  &\mathrm{in}\ \Omega,& 1 \leq i \leq n,
\end{array}
\right.
\]
satisfies $z_i(T,.)=0$ for every $1 \leq i \leq n$.
\end{theo}
\Cref{TheoCascadeDeux} is a particular case of \Cref{theolinear} with $m=1$ (see \Cref{SectionTheoLinear} below).
\begin{rmk}\label{RMKCarl}
The proof strategies of \Cref{TheoCascade} and \Cref{Guerrero} rely on global Carleman estimates. By combining these methods to obtain \Cref{TheoCascadeDeux}, we are facing the same difficulty as appearing in \cite{FCGBT}, i.e., we can only treat the case of $n\times n$ systems with $n \leq 5$. Inspired by the recent work \cite{LiZu} by Pierre Lissy and Enrique Zuazua based on the Lebeau-Robbiano method, we prove \Cref{TheoCascadeDeux}.
\end{rmk}
\begin{rmk}\label{RMKDdiag}
Note that the diffusion matrix \tiny $\begin{pmatrix}
d_{1} & 0 & \dots & \dots & 0\\
1 & d_{2} &\ddots &\ddots  &\vdots\\
0 & \ddots & \ddots & \ddots & \vdots\\
\vdots & \ddots & \ddots & \ddots & 0\\
0 & \dots &  0 &1 & d_{n}
\end{pmatrix}$ \normalsize is diagonalizable if and only if $d_i \neq d_j$ for every $i \neq j$. In this case, the Kalman condition \cite[Theorem 5.3]{AKBDGB3} (that can be easily extended to Neumann boundary conditions instead of Dirichlet boundary conditions) yields null-controllability for initial data in $L_{\sigma}^2(\Omega)^{n}$ which is smaller that $L^2(\Omega) \times L_{\sigma}^2(\Omega)^{n-1}$.
\end{rmk}
For a recent survey on the null-controllability of linear parabolic systems, see \cite{AKBGBT} and references therein.
\subsubsection{Semilinear results}
Let $(d_i) \in (0,+\infty)^n$, $(f_i)_{1 \leq i \leq n} \in C^{\infty}(\R^n;\R)^n$, satisfying $f_i(0) = 0$ for every $1 \leq i \leq n$. For semilinear parabolic systems 
\begin{equation}
\label{Semilinear}
\tag{NL}
\forall 1 \leq i \leq n,\ \left\{
\begin{array}{l l}
\partial_t z_i - d_i \Delta z_{i} =f_i((z_j)_{1 \leq j \leq n}) + h 1_{\omega} 1_{i \in \{1\}} &\mathrm{in}\ (0,T)\times\Omega,\\
\frac{\partial z_i}{\partial \nu} = 0 &\mathrm{on}\ (0,T)\times\partial\Omega,\\
z_i(0,.)=z_{i,0} &\mathrm{in}\  \Omega,
\end{array}
\right.
\end{equation}
the usual strategy consists in deducing a \textit{local null-controllability result} for \eqref{Semilinear} from a (global) null-controllability result for the \textit{linearized system} around $((\overline{z_i})_{1 \leq i \leq n}, \overline{h})=(0,0)$
\begin{equation}
\label{SemilinearLin}
\tag{L}
\forall 1 \leq i \leq n,\ \left\{
\begin{array}{l l}
\partial_t z_i - d_i \Delta z_{i} =\sum\limits_{j=1}^n \partial_j f_i(0) z_j + h_i 1_{\omega} 1_{i \in \{1\}}&\mathrm{in}\ (0,T)\times\Omega,\\
\frac{\partial z_i}{\partial \nu} = 0 &\mathrm{on}\ (0,T)\times\partial\Omega,\\
z_i(0,.)=z_{i,0} &\mathrm{in}\  \Omega.
\end{array}
\right.
\end{equation}
In this article, we use the powerful \textit{source term method}, introduced by Yuning Liu, Takéo Takahashi and Marius Tucsnak in \cite{LTT}. This method enables to prove null-controllability in time $T$ for
\begin{align}
\label{SemilinearLinSource}
\tag{L+S}
& \notag\left\{
\begin{array}{l l}
\forall 1 \leq i \leq n\\
\partial_t z_i - d_i \Delta z_{i} =\sum\limits_{j=1}^n \partial_j f_i(0) z_j + h_i 1_{\omega} 1_{i \in \{1\}} + S_i(t,x) &\mathrm{in}\ (0,T)\times\Omega,\\
\frac{\partial z_i}{\partial \nu} = 0 &\mathrm{on}\ (0,T)\times\partial\Omega,\\
z_i(0,.)=z_{i,0} &\mathrm{in}\  \Omega,
\end{array}
\right.
\end{align}
where $S_i$ ($1 \leq i \leq n$) has a prescribed decay rate at $t =T$, depending on the cost of null-controllability for \eqref{SemilinearLin}. Then, for $(z_{i,0})_{1 \leq i \leq n}$ sufficiently small (in an appropriate norm), a fixed-point strategy in suitable spaces is applied to the map: $$\mathcal{N} : (S_i)_{1 \leq i \leq n} \mapsto \left(f_i((z_j)_{1 \leq j \leq n})-\sum\limits_{j=1}^n\partial_j f_i(0)\right)_{1 \leq i \leq n},$$ where $(z_i)_{1 \leq i \leq n}$ is the solution associated to the optimal control (in an appropriate norm) of \eqref{SemilinearLinSource}. Consequently, the local null-controllability for \eqref{Semilinear} comes from the null-controllability of \textbf{only one} linear system \eqref{SemilinearLin}.\\
\indent In this article, we adapt the source term method in a $L^{\infty}$-context in the following way.
\begin{itemize}
\item The source term method in $L^2$ enables to prove a strong observability inequality (see \Cref{CorStrongObs}). This estimate looks like a global Carleman estimate (see for example \cite[Lemma 1.3]{FCG}), whereas the method to get it is very different.
\item By using the Penalized Hilbert Uniqueness Method, introduced by Viorel Barbu in \cite{B}, we construct $L^{\infty}$-controls (see \Cref{theolinearLinfty}).
\item We use another time the source term method in $L^{\infty}$ (see \Cref{ControlPropEstiLTTLinfty}).
\item We conclude by an appropriate inverse mapping theorem (see \Cref{Fixed point}).
\end{itemize}
For other results using the source term method, see for instance \cite{BM}, \cite{FCLdM} and \cite{MiTa}.\\
\indent Another strategy to get local controllability result for \eqref{Semilinear}, called the \textit{Small $L^{\infty}$-perturbations method}  is used in \cite{AKBD}, \cite{B}, \cite{LB}, \cite{LCMML} and \cite{WZ}. This method requires the null-controllability of a family of linear parabolic systems. Thus, this type of result is proved by using global Carleman estimates that enable to treat parabolic operators as $\partial_t z  - \Delta z  + a(t,x) z $ with $a \in L^{\infty}(Q_T)$.\\
\indent Nevertheless, for the linearized system around $((\overline{z_i})_{1 \leq i \leq n}, \overline{h})=(0,0)$ of \Cref{SectionLinearization} (see below), a technical difficulty appears when we want to prove an observability inequality for the adjoint system by using global Carleman estimates when $n>4$ (see \Cref{RMKCarl}).\\
\indent For degenerate cases (see \Cref{SectionDegen}), i.e., when \eqref{SemilinearLin} is not null-controllable, one can try to perform the \textit{return method}, introduced by Jean-Michel Coron in \cite{C2} (see also \cite[Chapter 6]{C}). This method consists in finding a reference trajectory $((\overline{z_i})_{1 \leq i \leq n}, \overline{h})$ verifying $z_{i}(0,.)=z_{i}(T,.)=0$ ($1 \leq i \leq n$) of \eqref{Semilinear} such that the linearized system of \eqref{Semilinear} around $((\overline{z_i})_{1 \leq i \leq n}, \overline{h})$ is null-controllable. By using the small $L^{\infty}$-perturbations method, we can obtain a local null-controllability result for \eqref{Semilinear}. See for instance \cite{CGR}, \cite{CGMR}, \cite{CG} and \cite{LB}.\\
\indent Let us also mention the new method of \cite{LB2} to prove the global null-controllability of reaction-diffusion systems of two species with only one control force by constructing controls of the heat equation behaving as odd regular functions.

\section{Definition and general properties of the trajectories}
In this section, we introduce the concept of \textit{trajectory} of \eqref{systNLC} which requires a well-posedness result (see \Cref{deftraj}).
\subsection{Usual notations}
\indent Let $k, l \in \N^*$. We denote by $\mathcal{M}_k(\R)$ (respectively $\mathcal{M}_{k,l}(\R)$) the algebra of matrices with $k$ lines and $k$ columns (respectively the algebra of matrices with $k$ lines and $l$ columns with entries in $\R$. The matrix $A^{\text{tr}} \in \mathcal{M}_{l,k}(\R)$ denotes the transpose of the matrix $A \in \mathcal{M}_{k,l}(\R)$. For $M \in \mathcal{M}_{k}(\R)$, $Sp(M)$ is the set of complex eigenvalues of $M$:
$Sp(M) := \{\lambda \in \C;\ \exists X \in \C^k\setminus\{0\},\ MX = \lambda X\}$.\\
\indent For $\tau > 0$, we introduce $$Q_{\tau} := (0,\tau)\times\Omega.$$
\indent For every $(a_1, \dots, a_n) \in \R^n$, we define 
\begin{align}
\label{deffi}
\forall 1 \leq i \leq n, \ f_i(a_1, \dots, a_n) &:=  (\beta_i-\alpha_i)\left(\prod\limits_{k=1}^n a_k^{\alpha_k} - \prod\limits_{k=1}^n a_k^{\beta_k}\right),\\
\label{defF}F(a_1, \dots, a_n) &:= (f_i(a_1, \dots, a_n))_{1 \leq i \leq n}^{\text{tr}}.
\end{align}
We introduce 
\begin{align}
\label{defU}U := (u_1, \dots, u_n)^{\text{tr}},\quad U^{*} := (u_1^{*}, \dots, u_n^{*})^{\text{tr}}.
\end{align}
\indent Up to a renumbering of $(u_i)_{1 \leq i \leq n}$, we can assume that $J = \{1, \dots, m\}$ where $J$ is defined in \eqref{DefJ}. Hence, we define \begin{equation}
\label{defHJ}H^{J} := (h_1 , \dots, h_{m} , 0, \dots, 0)^{\text{tr}}.
\end{equation}
\indent We must be careful with the dependence on the constants appearing in the estimates with respect to $T$ (when $T$ is small). That is why, from now and until the end of the article, we assume that 
\begin{equation}
T \in (0,1).
\label{Tsmall}
\end{equation}
Unless otherwise specified, we denote by $C$ various positive constants varying from line to line.
\subsection{Well-posedness results and definition of a trajectory}
We define the function space
\begin{equation}
W_T := L^2(0,T;H^1(\Omega)) \cap H^1(0,T;(H^1(\Omega))'),
\end{equation}
that satisfies the continuous embedding
\begin{equation}
\label{injclassique}
W_T \hookrightarrow C([0,T];L^2(\Omega)).
\end{equation}
The following result introduces the notion of solution for linear parabolic systems and provides estimates in terms of the initial data and the source term.
\begin{defprop}\label{wpl2linfty}
Let $k \in \N^*$, $D\in \mathcal{M}_{k}(\R)$ a diagonalizable matrix such that $ Sp(D) \subset (0,+\infty)$, $A\in \mathcal{M}_{k}(\R)$, $U_0 \in L^2(\Omega)^k$, $S \in L^{2}(Q_T)^k$. The following Cauchy problem admits a unique weak solution $U \in W_T^k $
\begin{equation}
\label{systLinearGeneral}
\left\{
\begin{array}{l l}
\partial_t U - D \Delta U= AU + S&\mathrm{in}\ (0,T)\times\Omega,\\
\frac{\partial U}{\partial \nu} = 0 &\mathrm{on}\ (0,T)\times\partial\Omega,\\
U(0,.)=U_0 &\mathrm{in}\  \Omega.
\end{array}
\right.
\end{equation}
This means that $U$ is the unique function in $W_T^k$ that satisfies the variational formulation
\begin{align}
&\forall V \in L^2(0,T;H^1(\Omega)^k),\\ &\int_0^T (\partial_t U ,V)_{(H^1(\Omega)^k)',H^1(\Omega)^k)} + \int_{Q_T} D \nabla U .  \nabla V = \int_{Q_T} (AU + S) . V\notag,
\label{formvar}
\end{align}
and
\begin{equation}
U(0,.) = U_0 \ \mathrm{in}\ L^2(\Omega)^k.
\label{condinitL2}
\end{equation}
Moreover, there exists $C>0$ independent of $U_0$ and $S$ such that
\begin{equation}
\norme{U}_{W_T^k} \leq C \left(\norme{U_0}_{L^{2}(\Omega)^k}+\norme{S}_{L^{2}(Q_T)^k}\right).
\label{estl2faible}
\end{equation}
Finally, if $U_0 \in L^{\infty}(\Omega)^k$ and $S \in L^{\infty}(Q_T)^k$, then $U \in L^{\infty}(Q_T)^k$ and there exists $  C >0$ independent of $U_0$ and $S$ such that
\begin{equation}
\norme{U}_{L^{\infty}(Q_T)^k} \leq C \left(\norme{U_0}_{L^{\infty}(\Omega)^k}+\norme{S}_{L^{\infty}(Q_T)^k}\right).
\label{estl2faiblelinfty}
\end{equation}
\end{defprop}
The proof of \Cref{wpl2linfty} can be found in \cite[Proposition 2.3]{LB}.\\
\indent The following result introduces the notion of trajectory associated to the nonlinear system \eqref{systNLC} (see \Cref{SectionControlOpQuestion}).
\begin{defprop}\label{deftraj}
Let $D = \text{diag}(d_1, \dots, d_n)$ with $d_i \in (0,+\infty)$. For every $U_0 \in L^{\infty}(\Omega)^n$, $(U,H^J)$ (see \eqref{defU} and \eqref{defHJ}) is a \textbf{trajectory} of \eqref{systNLC} if 
\begin{enumerate}[nosep]
\item $(U,H^J) \in \Big(W_T \cap L^{\infty}(Q_T)\Big)^n\times L^{\infty}(Q_T)^m $,
\item $U$ is the (unique) solution of \eqref{systNLC}. This means that $U$ is the unique function in $\Big(W_T \cap L^{\infty}(Q_T)\Big)^n$ which satisfies
\begin{align}
\label{formvarNLC}&\forall V \in L^2(0,T;H^1(\Omega)^k),\\
 &\int_0^T (\partial_t U ,V)_{(H^1(\Omega)^k)',H^1(\Omega)^k)} + \int_{Q_T} D \nabla U .  \nabla V = \int_{Q_T} \left(F\left(U\right) +  H^J 1_{\omega}\right). V,
\notag
\end{align}
with $F$ defined in \eqref{defF} and
\begin{equation}
U(0,.) = U_0 \ \mathrm{in}\ L^{\infty}(\Omega)^k.
\label{condinitLinfty}
\end{equation}
\end{enumerate}
Moreover, $(U,H^J)$ is a \textbf{trajectory} of \eqref{systNLC} \textbf{reaching} $U^{*}$ (see \eqref{defU}) in time $T$ if
\[U(T,.)=U^*.\]
\end{defprop}
The proof of the uniqueness of \Cref{deftraj} comes from the fact that $F$ is locally Lipschitz on $\R^{n}$ (see the proof of \cite[Definition-Proposition 2.4]{LB2}).
\subsection{Invariant quantities of the nonlinear dynamics}
In this section, we show that in the system \eqref{systNLC} (see \Cref{SectionControlOpQuestion}), some quantities are invariant. They impose some restrictions on the initial condition, for the controllability results.
\begin{prop}\label{PropInvQuant}
Let $U_0 \in L^{\infty}(\Omega)^n$ and let $(U,H^J)$ be a trajectory of \eqref{systNLC} reaching $U^*$ in time $T$. Then, we have for every $k \neq l \in \{m+1, \dots, n\}$, $t \in [0,T]$,  
\begin{equation}
\int_{\Omega} \frac{u_{k}(t,x) - u_{k}^{*}}{\beta_k-\alpha_k}dx= \int_{\Omega} \frac{u_{l}(t,x) - u_{l}^{*}}{\beta_l-\alpha_l}dx,
\label{ci1t}
\end{equation}
\begin{equation}
\Big(d_k =d_l\Big) \Rightarrow \left(\frac{u_{k}(t,.) - u_{k}^{*}}{\beta_k-\alpha_k}= \frac{u_{l}(t,.) - u_{l}^{*}}{\beta_l-\alpha_l}\right).
\label{lienciedt}
\end{equation}
In particular, for every $k \neq l \in \{m+1, \dots, n\}$,
\begin{equation}
\int_{\Omega} \frac{u_{k,0}(x) - u_{k}^{*}}{\beta_k-\alpha_k}dx= \int_{\Omega} \frac{u_{l,0}(x) - u_{l}^{*}}{\beta_l-\alpha_l}dx,
\label{ci1}
\end{equation}
\begin{equation}
\Big(d_k =d_l\Big) \Rightarrow \left(\frac{u_{k,0} - u_{k}^{*}}{\beta_k-\alpha_k}= \frac{u_{l,0}- u_{l}^{*}}{\beta_l-\alpha_l}\right).
\label{liencied}
\end{equation}
\end{prop}
The proof of \Cref{PropInvQuant} is done in \Cref{ProofPropinvquant}. We prove \eqref{ci1t} by integrating with respect to the space variable an appropriate linear combination of equations of \eqref{systNLC} and by using the Neumann boundary conditions. We prove \eqref{lienciedt} by the backward uniqueness of the heat equation applied to an appropriate linear combination of equations of \eqref{systNLC}.\\
\indent The equation \eqref{lienciedt} implies that we can reduce the number of components of $(u_i)_{1 \leq i \leq n}$ of \eqref{systNLC} when some diffusion coefficients $d_i$ are equal for $m+1 \leq i \leq n$. This simplify the study, thus in order to treat the more difficult case, we make the following hypothesis.
\begin{ass}
\label{Assdkneqdl}
For every $k \neq l \in \{m+1, \dots, n\}$, $d_k \neq d_l$. 
\end{ass}
\begin{rmk}
\label{MassCondeq}
It will be interesting to note that the mass condition \eqref{ci1} is equivalent to
\begin{equation}
\label{ci1Eq}
\forall k \geq m+2,\ \int_{\Omega} \frac{u_{k,0}(x) - u_k^{*}}{\beta_k - \alpha_k} dx = \int_{\Omega} \frac{u_{m+1,0}(x) - u_{m+1}^{*}}{\beta_{m+1} - \alpha_{m+1}} dx.
\end{equation}
\end{rmk}
\section{An adequate change of variables and linearization}
\subsection{Change of variables - Cross diffusion system}
\label{sectionChangeVar}
The goal of this section is to transform the controlled system \eqref{systNLC} (see \Cref{SectionControlOpQuestion}) satisfied by $U$ into another system of \textit{cascade} type for which we better understand the controllability properties. Roughly speaking, for $1 \leq i \leq m$, the component $u_i$ is easy to control thanks to the localized control term $h_i 1_{\omega}$. Thus, the challenge is to understand how the reaction term $f_i(U)$ (see \eqref{deffi}) acts on the component $u_i$ for $m+1 \leq i \leq n$.\\
\indent We multiply the $(m+1)$-th equation of \eqref{systNLC} by $1/((\beta_{m+1}-\alpha_{m+1})(d_{m+1}-d_{m+2}))$ and the $(m+2)$-th equation of \eqref{systNLC} by $1/((\beta_{m+2}-\alpha_{m+2})(d_{m+2}-d_{m+1}))$ and we sum:
$$ \partial_t v_{m+2} - d_{m+2} \Delta v_{m+2} = \frac{1}{\beta_{m+1}-\alpha_{m+1}}\Delta u_{m+1},$$
where 
$$ v_{m+2} = \frac{1}{(\beta_{m+1}-\alpha_{m+1})(d_{m+1}-d_{m+2})} u_{m+1} + \frac{1}{(\beta_{m+2}-\alpha_{m+2})(d_{m+2}-d_{m+1})} u_{m+2}.$$
Roughly speaking, this linear combination enables to “kill” the reaction-term and to create a coupling term of second order.\\
\indent By iterating this strategy, we construct a linear transformation $V=PU$ such that $u_{m+1}$ acts on $v_{m+2}$, $v_{m+2}$ acts on $v_{m+3}$, ..., $v_{n-1}$ acts on $v_{n}$ through cross diffusion terms. Moreover, we transform the problem of controllability for $U$ to $U^*$ into a null-controllability problem for 
\begin{equation}
Z:=P(U-U^{*}),
\label{defz}
\end{equation} 
where $P$ is the invertible triangular matrix defined by:
\begin{equation}\label{matrixchangvar}
P := 
\left(
\begin{array}{c |c}
I_{m} & (0)\\
\hline
(0) & *
\end{array}
\right),
\end{equation}
with
\begin{equation}
\label{DefcoeffP}
\forall k, l \geq m+1,\ P_{kl} :=
\renewcommand*{\arraystretch}{1.5}
\left\{
\begin{array}{c l}
\left((\beta_l - \alpha_l)\prod\limits_{\substack{m+1 \leq r \leq k \\ r \neq l}}(d_l-d_r)\right)^{-1} & \mathrm{if}\  k \geq l,\\
0 & \mathrm{if}\  k <l,
\end{array}
\right.
\end{equation}
with the convention $\prod\limits_{\emptyset} = 1$.
\begin{prop}\label{propchangevarZ}
The couple $(U,H^J)$ is a trajectory of \eqref{systNLC} if and only if $(Z,H^J)$ satisfies
\small
\begin{align}
\left\{
\begin{array}{l l l}
\normalsize \partial_t z_i - d_i \Delta z_{i} = f_i(P^{-1}Z+U^{*}) + h_i 1_{\omega} &\mathrm{in}\ (0,T)\times\Omega, &\small{1 \leq i \leq m}\\
\partial_t z_{m+1} - d_{m+1} \Delta z_{m+1} = \frac{f_{m+1}(P^{-1}Z+U^{*})}{\beta_{m+1}-\alpha_{m+1}} &\mathrm{in}\ (0,T)\times\Omega,&\\
\normalsize \partial_t z_i - d_i \Delta z_{i} = \Delta z_{i-1}&\mathrm{in}\ (0,T)\times\Omega,& \small{m+2 \leq i \leq n}, \\
\frac{\partial Z}{\partial \nu} = 0 &\mathrm{on}\ (0,T)\times\partial\Omega,&\\
Z(0,.)=Z_0 &\mathrm{in}\  \Omega,&
\end{array}
\right.
\label{systNLCZ1}
\end{align}
\normalsize
\end{prop}
The proof of \Cref{propchangevarZ} is done in \Cref{Eq2systems}.\\
\indent We introduce the notations: for every $(a_1, \dots, a_n) \in \R^n$,
\begin{equation}
\label{defgi}
\forall 1 \leq i \leq m,\ g_i(a_1, \dots, a_n) := f_i(P^{-1}(a_1, \dots, a_n)^{\text{tr}}+U^*),
\end{equation}
\begin{equation}
\label{defgm+1}
g_{m+1}(a_1, \dots, a_n) := \frac{f_{m+1}(P^{-1}(a_1, \dots, a_n)^{\text{tr}}+U^*)}{\beta_{m+1}-\alpha_{m+1}},
\end{equation}
\begin{equation}
\label{defGgi}
G(a_1, \dots, a_n) := (g_1(a_1, \dots, a_n), \dots, g_{m+1}(a_1, \dots, a_n), 0 \dots, 0)^{\text{tr}},
\end{equation}
\begin{equation}
\label{defNewD}
D_J :=
\left(
\begin{array}{c |c}
diag(d_1, \dots, d_m) & (0)\\
\hline
(0) & D_{\sharp}
\end{array}
\right),\ 
D_{\sharp} := 
\begin{pmatrix}
d_{m+1} & 0 & \dots & \dots & 0\\
1 & d_{m+2} &\ddots &\ddots  &\vdots\\
0 & \ddots & \ddots & \ddots & \vdots\\
\vdots & \ddots & \ddots & \ddots & 0\\
0 & \dots &  0 &1 & d_{n}
\end{pmatrix}.
\end{equation}
With these notations, the nonlinear system \eqref{systNLCZ1} is
\begin{equation}
\label{defsystNLZ}
\tag{NL-Z}
\left\{
\begin{array}{l l}
 \partial_t Z - D_J \Delta Z = G(Z) + H^{J} 1_{\omega}&\mathrm{in}\ (0,T)\times\Omega,\\
\frac{\partial Z}{\partial \nu} = 0 &\mathrm{on}\ (0,T)\times\partial\Omega,\\
Z(0,.)=Z_{0} &\mathrm{in}\  \Omega.
\end{array}
\right.
\end{equation}
\subsection{Linearization}
\label{SectionLinearization}
We will work under the following hypothesis which will guarantee the null-controllability of this linearized system.
\begin{ass}
\label{asscouplordre0}
We assume that there exists $1 \leq j \leq m$ such that
\begin{align}
\label{couplordre0}
\partial_{j} g_{m+1}\Big(0, \dots, 0\Big)\neq 0.
\end{align}
By using \eqref{matrixchangvar} and \eqref{defgm+1}, we check that \eqref{couplordre0} is equivalent to
\begin{align}
\label{eqcouplordre0}
\partial_{j} f_{m+1}\Big(u_1^*, \dots, u_n^*\Big)\neq 0.
\end{align}
When $\alpha_j, \beta_j \geq 1$, a sufficient condition to ensure \eqref{eqcouplordre0} is
\begin{equation}
\label{GenericCondition}
\forall 1 \leq k \leq n,\ u_{k}^{*} \neq 0.
\end{equation}
Indeed, by using \eqref{deffi}, \eqref{Stat} and $\alpha_j \neq \beta_j$, if \eqref{GenericCondition} holds true then
\begin{align*}
\partial_{j} f_{m+1}\Big(u_1^*, \dots, u_n^*\Big) &= \alpha_j (u_j^{*})^{\alpha_j - 1} \prod\limits_{\substack{k=1 \\ k \neq j}}^n {u_k^{*}}^{\alpha_k} - \beta_j (u_j^{*})^{\beta_j - 1} \prod\limits_{\substack{k=1 \\ k \neq j}}^n {u_k^{*}}^{\beta_k}\\
& = \frac{\alpha_j}{u_j^{*}} \prod\limits_{k=1}^n {u_k^{*}}^{\alpha_k} - \frac{\beta_j}{u_j^{*}} \prod\limits_{k=1}^n {u_k^{*}}^{\beta_k}\\
& = \frac{\alpha_j-\beta_j}{u_j^{*}} \prod\limits_{k=1}^n {u_k^{*}}^{\alpha_k} \neq 0.
\end{align*}
Note that \eqref{GenericCondition} is not equivalent to \eqref{eqcouplordre0} as shown by the examples in \Cref{Appu1u3u2u4}.
\end{ass}
The linearized system of \eqref{defsystNLZ} satisfied by $Z$ around $(0,0)$ is
\begin{equation}
\label{defsystZ}
\tag{L-Z}
\left\{
\begin{array}{l l}
 \partial_t Z - D_J \Delta Z = A_J Z + H^{J} 1_{\omega}&\mathrm{in}\ (0,T)\times\Omega,\\
\frac{\partial Z}{\partial \nu} = 0 &\mathrm{on}\ (0,T)\times\partial\Omega,\\
Z(0,.)=Z_{0} &\mathrm{in}\  \Omega,
\end{array}
\right.
\end{equation}
where
\begin{equation}
\label{defcouplA}
A_{J} = (a_{ik})_{1 \leq i,k\leq n}, \qquad a_{ik} = \left\{
\begin{array}{c l}
\partial_k g_i(0, \dots, 0)& \mathrm{if}\  1 \leq i \leq m+1,\\
0 & \mathrm{if}\  m+2 \leq i \leq n.
\end{array}
\right.
\end{equation}
Up to a renumbering of the first $m$ equations, we can assume that $j=m$. Then, by \Cref{asscouplordre0}, we have
\begin{equation}
\label{CouplNonZeroA}
a_{m+1,m} \neq 0.
\end{equation}
Roughly speaking, we summarize the controllability properties established in the following diagram:
\begin{align*}
&h_1 \xrightarrow[]{controls} z_1,\ h_2 \xrightarrow[]{controls} z_2,\ \dots\ ,\ h_{m-1}\xrightarrow[]{controls} z_{m-1}, \\ 
&h_m \xrightarrow[]{controls} z_m  \xrightarrow[a_{m+1m} z_m]{controls}z_{m+1} \xrightarrow[{\Delta z_{m+1}}]{controls} z_{m+2}\xrightarrow[{\Delta z_{m+2}}]{controls} \dots  \xrightarrow[{\Delta z_{n-1}}]{controls} z_n.
\end{align*}

\section{Main results}

The goal of this section is to state the main results of the paper. First, we prove a local null-controllability result for the system \eqref{defsystNLZ} (see \Cref{sectionChangeVar}). Then, we deduce a local controllability result around $U^{*}$ for \eqref{systNLC} (see \Cref{SectionControlOpQuestion}).\\
\indent We have seen in \Cref{PropInvQuant} that a trajectory $(U, H^{J})$ reaching $U^{*}$ has to verify the condition \eqref{ci1t}. Thus, it prevents local-controllability from happening for arbitrary initial data. This is why we introduce a notion of local controllability adapted to \eqref{ci1}.\\
\indent Let $p \in [1,+\infty]$. We introduce the following subspace of $L^p(\Omega)^n$:
\begin{equation}
\label{defL_inv^2}
L_{inv}^p := \left\{ Z_0 \in L^p(\Omega)^n\ ;\ \forall m+2 \leq i \leq n,\ \int_{\Omega} z_{i,0}(x) dx = 0\right\}.
\end{equation}
\begin{theo}
\label{LocContrLinfty}
Under \Cref{Assdkneqdl} and \Cref{asscouplordre0}, the system \eqref{defsystNLZ} is locally null-controllable, i.e., there exists $r > 0$ such that for every $Z_0 \in L_{inv}^{\infty}$ verifying $\norme{Z_0}_{L^{\infty}(\Omega)^n} \leq r$, there exists $H^J \in L^{\infty}(Q_T)^m$ such that the solution $Z$ of \eqref{defsystNLZ} satisfies $Z(T,.)=0$.
\end{theo}
We deduce from \Cref{LocContrLinfty} the following local controllability result.
\begin{theo}
\label{mainresult1}
Under \Cref{Assdkneqdl} and \Cref{asscouplordre0}, the system \eqref{systNLC} is locally controllable around $U^{*}$, i.e., there exists $r > 0$ such that for every $U_0 \in L^{\infty}$ satisfying the mass condition \eqref{ci1} and $\norme{U_0 - U^*}_{L^{\infty}(\Omega)} \leq r$, there exists $H^J \in L^{\infty}(Q_T)^m$ such that the solution $U$ of \eqref{systNLC} satisfies $U(T,.)=U^{*}$.
\end{theo}
\Cref{LocContrLinfty} is equivalent to \Cref{mainresult1}. Indeed, it comes from \Cref{propchangevarZ} and the following equivalence
\begin{equation}
\label{InvquantZU}
Z_0 \in L_{inv}^{\infty} \Leftrightarrow U_0\ \text{satisfies}\ \eqref{ci1} \Leftrightarrow U_0\ \text{satisfies}\ \eqref{ci1Eq}\ (\text{\Cref{MassCondeq}}) .
\end{equation}
The proof of \eqref{InvquantZU} is done in \Cref{Proofinvquantzu}.
\begin{app}\label{Appu1u3u2u4}
For $n=4$, $\alpha_1 = \alpha_3 = \beta_2 = \beta_4 = 1$ and $\alpha_2 = \alpha_4 = \beta_1= 0$, we are studying the following system:
\small
\begin{equation}\label{systNLCu1u3u2u4}
\forall 1 \leq i \leq 4,\ \left\{
\begin{array}{l l}
\partial_t u_i - d_i \Delta u_{i} = (-1)^{i}(u_1 u_3 - u_2 u_4) + h_i 1_{\omega} 1_{i \in J} &\mathrm{in}\ (0,T)\times\Omega,\\
\frac{\partial u_i}{\partial \nu} = 0 &\mathrm{on}\ (0,T)\times\partial\Omega,\\
u_i(0,.)=u_{i,0} &\mathrm{in}\  \Omega.
\end{array}
\right.
\end{equation}
\normalsize
In this case, we check that \Cref{asscouplordre0} is
$$ \text{for}\ J= \{1,2,3\},\quad\Big(\exists j \in \{1,2,3\},\ \partial_j f_4(u_1^{*}, \dots, u_4^{*}) \neq 0\Big) \Leftrightarrow \Big((u_1^{*}, u_3^{*}, u_4^{*}) \neq (0,0,0)\Big),$$
$$\text{for}\ J=\{1,2\}, \quad  \Big(\exists j \in \{1,2\},\ \partial_j f_3(u_1^{*}, \dots, u_4^{*}) \neq 0\Big) \Leftrightarrow  \Big((u_3^{*}, u_4^{*}) \neq (0,0)\Big),$$
$$\text{for}\ J=\{1\}, \quad  \Big(\partial_1 f_2 (u_1^{*}, \dots, u_4^{*}) \neq 0\Big) \Leftrightarrow \Big(u_3^{*} \neq 0\Big).$$
Thus, \Cref{mainresult1} recovers the result of \cite[Theorem 3.2]{LB} except for the case $J= \{1,2,3\}$ and $(u_1^{*}, u_3^{*}, u_4^{*}) =(0,0,0)$ that the proof of the present article does not treat (see \Cref{RMArt1} for more details about the strategy of \cite{LB}).
\end{app}

\section{Linear null-controllability under constraints in $L^2$}\label{SectionTheoLinear}
The main result of this section, stated in the following theorem, is the null-controllability in $L_{inv}^2$ for the linear system \eqref{defsystZ} (see \Cref{SectionLinearization}).
\begin{theo}
\label{theolinear}
The system \eqref{defsystZ} is null-controllable in $L_{inv}^2$. More precisely, there exists $C>0$ such that for every $T>0$ and $Z_0 \in L_{inv}^2$, there exists a control $H^J \in L^2((0,T)\times \Omega)^m$ verifying
\begin{equation}
\label{esticontrolTh}
\norme{H^J}_{L^2(Q_T)^m} \leq C_T \norme{Z_0}_{L^2(\Omega)^n},\ \text{where}\ C_T =Ce^{C/T},
\end{equation}
and such that the solution $Z \in W_T^n$ of \eqref{defsystZ} satisfies $Z(T,.)=0$.
\end{theo}
The goal of the next two subsections is to prove \Cref{theolinear}. The proof is based on the Lebeau-Robbiano's method, introduced for the first time to prove the null-controllability of the heat equation (see \cite{LR}). First, it consists in establishing a null-controllability result in finite dimensional subspaces of $L_{inv}^2$ with a precise estimate of the cost of the control (see \Cref{TheoLinearLowFreq}). This first step is based on two main results: the spectral inequality for eigenfunctions of the Neumann-Laplace operator (see \Cref{lemSpecIn}) and precise observability estimates of linear finite dimensional systems associated to the adjoint system of \eqref{defsystZ} (see \Cref{lemmaCobs}). Secondly, we conclude by a time-splitting procedure: the control $H^J$ is built as a sequence of active controls and passive controls. The passive mode allows to take advantage of the natural parabolic exponential decay of the $L^2$ norm of the solution. This decay enables to compensate the cost of the control which steers the low frequencies to $0$ (see \Cref{LebeauRobMeth}).\\

\subsection{A null-controllability result for the low frequencies} 
We define $H_{Ne}^2(\Omega) := \left\{y \in H^2(\Omega)\ ;\ \frac{\partial y}{\partial \nu} = 0\right\}$. The unbounded operator on $L^2(\Omega)$: $(-\Delta,H_{Ne}^2(\Omega))$ is self-adjoint and has compact resolvent. Thus, we introduce the orthonormal basis $(e_k)_{k \geq 0}$ of $L^2(\Omega)$ of eigenfunctions associated to the increasing sequence of eigenvalues $(\lambda_k)_{k \geq 0}$ of the Laplacian operator, i.e., we have $-\Delta e_k = \lambda_k e_k$ and $(e_k,e_l)_{L^2(\Omega)} = \delta_{k,l}$. For $\lambda>0$, we define the finite dimensional space $E_{\lambda} = \left\{ \sum\limits_{\lambda_k \leq \lambda} c_k e_k\ ; \ c_k \in \R^n \right\} \subset L^2(\Omega)^n$ and the orthogonal projection $\Pi_{E_{\lambda}}$ onto $E_{\lambda}$ in $L^2(\Omega)^n$.\\

\indent The goal of this section is to prove the following null-controllability result in a finite dimensional subspace of $L_{inv}^2$.
\begin{prop} \label{TheoLinearLowFreq}
There exist $C >0$, $p_1 \in \N$ such that for every $\tau \in (0,T)$, $\lambda>0$, $Z_0 \in E_{\lambda} \cap L_{inv}^2$, there exists a control function $H^J \in L^2(Q_{\tau})$ verifying 
\begin{equation}
\norme{H^J}_{L^2(Q_\tau)^m}^2 \leq \frac{C}{\tau^{p_1}} e^{C \sqrt{\lambda}} \norme{Z_0}_{L^2(\Omega)^n}^2,
\label{CoutControlL^2BF}
\end{equation}
such that the solution $Z$ of 
\begin{equation}
\label{defsystZlowFreq}
\left\{
\begin{array}{l l}
 \partial_t Z - D_J \Delta Z = A_J Z + H^{J} 1_{\omega}&\mathrm{in}\ (0,\tau)\times\Omega,\\
\frac{\partial Z}{\partial \nu} = 0 &\mathrm{on}\ (0,\tau)\times\partial\Omega,\\
Z(0,.)=Z_{0} \in E_{\lambda} &\mathrm{in}\  \Omega,
\end{array}
\right.
\end{equation}
satisfies $Z(\tau,.)=0$.
\end{prop}
From \Cref{TheoLinearLowFreq}, for every $\tau,\lambda>0$ and $Z_0 \in E_{\lambda} \cap L_{inv}^2$  we introduce the notation:
\begin{equation}
\label{ControleMinL^2}
H_{\lambda}(Z_0, 0,\tau) := H^J,
\end{equation}
such that the solution $Z$ of \eqref{defsystZlowFreq} satisfies $Z(\tau,.)=0$ and $H^J$ is the minimal-norm element of $L^2(Q_{\tau})^m$ satisfying the estimate \eqref{CoutControlL^2BF}. In other words, $H^J$ is the projection of $0$ in the nonempty closed convex set of controls satisfying \eqref{CoutControlL^2BF} and driving the solution $Z$ of \eqref{defsystZlowFreq} in time $\tau$ to $0$.\\

By the Hilbert Uniqueness Method (see \cite[Theorem 2.44]{C}), in order to prove \Cref{TheoLinearLowFreq}, we need to prove an observability inequality for the solution of the adjoint system of \eqref{defsystZlowFreq}.
\begin{prop}
\label{propobsLowfreq}
There exist $C>0$, $p_1 \in \N$ such that for every $\tau \in (0,T)$, $\lambda>0$ and $\varphi_{\tau} \in E_{\lambda} \cap L_{inv}^2$, the solution $\varphi$ of 
\begin{equation}
\left\{
\begin{array}{l l}
-\partial_t {\varphi} - D_J^{\text{tr}} \Delta \varphi = A_J^{\text{tr}} \varphi &\mathrm{in}\ (0,\tau)\times\Omega,\\
\frac{\partial\varphi}{\partial \nu} = 0 &\mathrm{on}\ (0,\tau)\times\partial\Omega,\\
\varphi(\tau,.)=\varphi_{\tau} &\mathrm{in}\  \Omega,
\end{array}
\right.
\label{adjLowFreq}
\end{equation}
satisfies
\begin{equation}
\label{obslowFreq}
\norme{\varphi(0,.)}_{L^2(\Omega)^n}^2 \leq \frac{C}{\tau^{p_1}} e^{C \sqrt{\lambda}} \sum\limits_{i=1}^{m} \int_{0}^{\tau}\int_{\omega} \left| \varphi_{i}(t,x)\right|^{2} dxdt.
\end{equation}
\end{prop}
\begin{proof}
The proof is inspired by \cite[Section 3]{LiZu}.\\
\indent Let $\tau >0, \lambda >0$ and $\varphi_{\tau} \in E_{\lambda} \cap L_{inv}^2$. We have:
\begin{equation}
\label{decomInit}
\varphi_{\tau}(x) = \sum_{\lambda_k \leq \lambda} \varphi_{k}^{\tau} e_k(x),
\end{equation}
with $\varphi_{k}^{\tau} \in F_k$ where $F_0 := \R^{m+1} \times \{0\}^{n-m-1}$ because $\varphi_{\tau} \in L_{inv}^2$ and $F_k := \R^{n}$ for $k \geq 1$.\\
\indent Then, the solution $\varphi$ of \eqref{adjLowFreq} is
\begin{equation}
\label{decompsol}
\forall (t,x) \in (0,\tau)\times\Omega,\ \varphi(t,x) = \sum\limits_{\lambda_k \leq \lambda} \varphi_{k}(t) e_k(x),
\end{equation}
where $\varphi_k$ is the unique solution of the ordinary differential system
\begin{equation}
\label{systedo}
\left\{
\begin{array}{ll}
-\varphi_k' + \lambda_k D_{J}^{\text{tr}} \varphi_k= A_J^{\text{tr}}\varphi_k,& \text{in}\ (0,\tau),\\
\varphi_k(\tau)=\varphi_{k}^{\tau}&.
\end{array}
\right.
\end{equation}
We recall the spectral inequality for eigenfunctions of the Neumann-Laplace operator.
\begin{lem}\cite[Theorem 14.6]{JeLe}\\
\label{lemSpecIn}
There exists $C>0$ such that for every sequence $(a_k)_{ k\geq 0} \subset \C^{\N}$ and for every $\lambda >0$, we have:
\begin{equation}
\label{spectralin}
\sum\limits_{\lambda_k \leq \lambda} |a_k|^{2} = \int_{\Omega} \left|\sum\limits_{\lambda_k \leq \lambda} a_k e_k(x)\right|^{2} dx \leq C e^{C \sqrt{\lambda}} \int_{\omega} \left|\sum\limits_{\lambda_k \leq \lambda} a_k e_k(x)\right|^{2} dx.
\end{equation}
\end{lem}
By using \eqref{spectralin} for $a_k=\varphi_{k,i}(t)$ with $1 \leq i \leq m$ and by summing on $1\leq i\leq m$, we obtain that there exists $C >0$ such that
\begin{equation}
\label{spectralinbis}
\sum\limits_{\lambda_k \leq \lambda} \sum\limits_{i=1}^{m} |\varphi_{k,i}(t)|^{2} \leq C e^{C \sqrt{\lambda}} \sum\limits_{i=1}^{m} \int_{\omega} \left|\sum\limits_{\lambda_k \leq \lambda}  \varphi_{k,i}(t) e_k(x)\right|^{2} dx.
\end{equation}
By integrating with respect to the time variable between $0$ and $\tau$ the inequality \eqref{spectralinbis}, we obtain
\begin{equation}
\label{spectralinfour}
\int_{0}^{\tau}\sum\limits_{\lambda_k \leq \lambda} \sum\limits_{i=1}^{m} |\varphi_{k,i}(t)|^{2}dt \leq C e^{C \sqrt{\lambda}} \sum\limits_{i=1}^{m} \int_{0}^{\tau}\int_{\omega} \left|\sum\limits_{\lambda_k \leq \lambda}  \varphi_{k,i}(t) e_k(x)\right|^{2} dxdt.
\end{equation}
Now, our goal is to establish the following lemma.
\begin{lem}
\label{lemmaCobs}
There exist $C$, $p_1, p_2 \in \N$ such that for every $\tau \in (0,1)$, $k \in \N$, $\varphi_{k}^{\tau} \in F_k$, the solution $\varphi_k$ of \eqref{systedo} satisfies
\begin{equation}
\label{obsap}
\norme{\varphi_{k}(0)}^2 \leq C \left(1 + \frac{1}{\tau^{p_1}} + \lambda_k^{p_2}\right) \sum\limits_{i=1}^{m}\int_{0}^{\tau} |\varphi_{k,i}(t)|^{2} dt.
\end{equation}
\end{lem}
One can take for instance $p_1 =p_2 = 2(n-m+1/2)$ in \eqref{obsap}, we give a proof in \Cref{ProofObsDimfinie} (see also \cite{Se}).\\
\indent By using \eqref{spectralinfour}, \eqref{obsap}, we deduce that
\begin{align}
\label{finobs}
\sum\limits_{\lambda_k \leq \lambda} \norme{\varphi_{k}(0)}^{2} & \leq \sum\limits_{\lambda_k\leq \lambda} \frac{C}{\tau^{p_1}} (1 +  \lambda_k^{p_2}) \sum\limits_{i=1}^{m} \int_{0}^{\tau} |\varphi_{k,i}(t)|^{2} dt\\
& \leq \frac{C}{\tau^{p_1}} e^{C \sqrt{\lambda}}\sum\limits_{i=1}^{m} \int_{0}^{\tau}\int_{\omega} \left|\sum\limits_{\lambda_k \leq \lambda}  \varphi_{k,i}(t) e_k(x)\right|^{2} dxdt.\notag
\end{align}
By using \eqref{decompsol}, we deduce \eqref{obslowFreq} from \eqref{finobs}.
\end{proof}

\subsection{The Lebeau-Robbiano's method}
\label{LebeauRobMeth}
The goal of this section is to prove \Cref{theolinear}.
\begin{proof}
The proof is inspired by \cite[Section 6.2]{LLR}. The constants $C,C'$ will increase from line to line.\\
\indent We split the interval $[0,T] = \cup_{k \in \N} [a_k, a_{k+1}]$ with $a_0 =0$, $a_{k+1} = a_{k} + 2 T_k$ and $T_k = T/2^{k}$ for $k\in \N$. We also define $\mu_k = M2^{2k}$ for $M>0$ sufficiently large which will be defined later and for $k \in \N$. Then, we define the control $H^{J}$ in the following way:
\begin{itemize}
\item if $t \in (a_k, a_k+T_k)$, $H^{J} = H_{\mu_k}(\Pi_{E_{\mu_{k}}}Z(a_k,.),a_k, T_k)$ (see the notation \eqref{ControleMinL^2}) and $Z(t,.) = S(t-a_k)Z(a_k,.) + \int_{a_k}^{t} S(t-s) H^{J}(s,.)ds$,
\item if $t \in (a_k+T_k, a_{k+1})$, $H^{J} = 0$ and $Z(t,.) = S(t-a_k-T_k)Z(a_k+T_k,.)$,
\end{itemize}
where $S(t)$ denotes the semigroup of the parabolic system: $S(t) = e^{t (D_J \Delta + A_J)}$. In particular, by \eqref{estl2faible} and \eqref{injclassique}, $\norme{S(t)}_{\mathcal{L}(L^2(\Omega)^n)} \leq C$.\\
\indent By \eqref{CoutControlL^2BF}, the choice of $H^J$ during the interval time $[a_k, a_k +T_k]$ implies
\begin{align}
\label{SolActiv}
\norme{Z(a_k+T_k,.)}_{L^2(\Omega)^n}^2 &\leq (C + C (2^{-k} T)^{- p_1} e^{C \sqrt{M} 2^{k}}) \norme{Z(a_k,.)}_{L^2(\Omega)^n}^2\\
& \leq  \frac{C}{T^{p_1}} e^{C \sqrt{M}2^{k}} \norme{Z(a_k,.)}_{L^2(\Omega)^n}^2.\notag
\end{align}
During the passive period of the control, $t \in [a_k + T_k, a_{k+1}]$, the solution exponentially decreases:
\begin{equation}
\label{SolPassiv}
\norme{Z(a_{k+1},.)}_{L^2(\Omega)^n}^2 \leq C' e^{-C'M2^{2k} T_k}  \norme{Z(a_k+T_k,.)}_{L^2(\Omega)^n}^2.
\end{equation}
Thus, by using $2^{2k} T_k = 2^{k}T$, \eqref{SolActiv} and \eqref{SolPassiv}, we have
$$\norme{Z(a_{k+1},.)}_{L^2(\Omega)^n}^2 \leq  e^{C \sqrt{M} 2^{k}- C' M 2^{k}T + C/T }\norme{Z(a_k,.)}_{L^2(\Omega)^n}^2,$$
and consequently,
\begin{align}
\label{estglobalsol}
\norme{Z(a_{k+1},.)}_{L^2(\Omega)^n}^2 &\leq e^{\sum_{j=0}^{k} \left(C\sqrt{M} 2^{j}-C'M T 2^{j} + C/T\right)} \norme{Z_0}_{L^2(\Omega)^n}^2\\
& \leq e^{(C\sqrt{M} - C'MT)2^{k+1} + (C /T) (k+1)}\norme{Z_0}_{L^2(\Omega)^n}^2.\notag
\end{align} 
By taking $M$ such that $C \sqrt{M} - C' MT < 0$, for instance $M \geq 2(C/C'T)^{2}$, we conclude by \eqref{estglobalsol} that we have $ \lim_{k \rightarrow + \infty} \norme{Z(a_k,.)} = 0$, i.e., $Z(T,.)=0$ because $t \mapsto Z(t,.) \in C([0,T];L^2(\Omega)^n)$ because $H^J \in L^2(Q_T)^m$ (see \Cref{wpl2linfty} and \eqref{injclassique}) as we will show now.\\
\indent We have $\norme{H^J}_{L^2(Q_T)^m}^2 = \sum_{k=0}^{+\infty} \norme{H^J}_{L^2((a_k,a_k+T_k)\times\Omega)^m}^2$. Then, by using the estimate \eqref{CoutControlL^2BF} of the control on each time interval $(a_0,a_0+T_0)$ and the estimate \eqref{estglobalsol}, we get:
\begin{align}
\label{EsticontrolT}
\norme{H^J}_{L^2(Q_T)^m}^2 &\leq \left(C T_0^{-p_1} e^{C\sqrt{M}} + \sum\limits_{k \geq 1} C T_k^{-p_1} e^{C \sqrt{M}2^k} e^{-C'MT 2^{(k-1)}}\right) \norme{Z_0}_{L^2(\Omega)^n}^2\\
& \leq \left(C T^{-p_1} e^{C\sqrt{M}} + \sum\limits_{k \geq 1} C (2^{k}T^{-1})^{p_1} e^{(C \sqrt{M}-C'MT/2) 2^{k}}\right) \norme{Z_0}_{L^2(\Omega)^n}^2.\notag
\end{align}
By taking $M$ such that $C \sqrt{M}-C'MT/2 < 0$, for instance $M \geq 8 (C/C'T)^{2} \Rightarrow C \sqrt{M}-C'MT/2 = -C''/T$ with $C''>0$, we deduce from \eqref{EsticontrolT} that $H^J \in L^2(Q_T)^m$ and
\begin{align*}
\norme{H^J}_{L^2(Q_T)^m}^2 &\leq \left(C e^{C/T} + C \int_{0}^{+\infty} \left(\frac{\sigma}{T}\right)^{p_1} e^{-C'' \frac{\sigma}{T}} d\sigma\right) \norme{Z_0}_{L^2(\Omega)^n}^2\\
& \leq C e^{C/T} \norme{Z_0}_{L^2(\Omega)^n}^2.
\end{align*}
This concludes the proof of \Cref{theolinear}.
\end{proof}

\section{The source term method in $L^2$}\label{SourcetermL^2}
We use the source term method, introduced by Yuning Liu, Takéo Takahashi and Marius Tucsnak in \cite[Proposition 2.3]{LTT} to deduce a local null-controllability result for a nonlinear system from the null-controllability result for only one linear system (and an estimate of the cost of the control) (see also \cite{BM}).\\
\indent By \Cref{theolinear}, we have an estimate for the control cost in $L^2$, then we fix $M>0$ such that $C_T \leq M e^{M/T}$. Let $q \in (1, \sqrt{2})$ and $p > q^{2}/(2-q^{2})$. We define the weights
\begin{equation}
\label{defrho0bis}
\rho_0(t) := M^{-p} \exp\left(- \frac{Mp}{(q-1)(T-t)}\right),
\end{equation}
\begin{equation}
\label{defrhoG}
\rho_{\mathcal{S}}(t) = M^{-1-p} \exp\left(-\frac{(1+p)q^{2}M}{(q-1)(T-t)}\right).
\end{equation}
\begin{rmk}
\label{rmkp>q^2}
The assumption $ p > q^{2}/(2-q^{2})\Leftrightarrow 2p > (1+p)q^{2}$ implies
\begin{equation}
\label{rhorhoG}
\rho_0^{2}/\rho_{\mathcal{S}} \in C([0,T]),
\end{equation}
which will be useful for the estimate of the polynomial nonlinearity (see \Cref{Fixed point}).\\
\end{rmk}
\indent Let $r \in \{2, +\infty\}$. For $S \in L^r((0,T);L_{inv}^r)$, $H^J \in L^r((0,T);L^r(\Omega)^m)$, $Z_0 \in L_{inv}^r$, we introduce the following system:
\begin{equation}
\tag{L+S-Z}
\label{defsystZSourceS}
\left\{
\begin{array}{l l}
 \partial_t Z - D_J \Delta Z = A_J Z + S + H^{J} 1_{\omega}&\mathrm{in}\ (0,T)\times\Omega,\\
\frac{\partial Z}{\partial \nu} = 0 &\mathrm{on}\ (0,T)\times\partial\Omega,\\
Z(0,.)=Z_{0} &\mathrm{in}\  \Omega.
\end{array}
\right.
\end{equation}
Then, we define associated spaces for the source term, the state and the control
\begin{align}
& \mathcal{S}_r := \left\{S \in  L^r((0,T);L_{inv}^r)\ ;\ \frac{S}{\rho_{\mathcal{S}}} \in  L^r((0,T);L_{inv}^r)\right\},\label{defpoidsg}\\
& \mathcal{Z}_r := \left\{Z \in L^r((0,T);L_{inv}^r)\ ;\ \frac{Z}{\rho_0} \in L^r((0,T);L_{inv}^r)\right\},
\label{defpoidsf}\\
& \mathcal{H}_r := \left\{H^J \in L^r((0,T);L^r(\Omega)^m)\ ;\ \frac{H^J}{\rho_{0}}  \in L^r((0,T);L^r(\Omega)^m)\right\}.\label{defpoidsu}
\end{align}
\begin{rmk}
From the behaviors near $t=T$ of $\rho_{\mathcal{S}}$ and $\rho_0$, we deduce that each element of $\mathcal{S}_r$, $\mathcal{Z}_r$, $\mathcal{H}_r$ vanishes at $t=T$.
\end{rmk}
From the abstract result: \cite[Proposition 2.3]{LTT}, we deduce the null-controllability for \eqref{defsystZSourceS} in $L_{inv}^2$.
\begin{prop}\label{PropSourceTermL^2}
For every $S \in \mathcal{S}_2$ and $Z_0 \in L_{inv}^2$, there exists $H^J \in \mathcal{H}_2$, such that the solution $Z$ of \eqref{defsystZSourceS} satisfies $Z \in \mathcal{Z}_2$. Furthermore, there exists $C >0$, not depending on $S$ and $Z_0$, such that
\begin{equation}
\label{EstiLTT}
\norme{Z/\rho_0}_{C([0,T];L^2(\Omega)^n)} + \norme{H^J}_{\mathcal{H}_2} \leq C_T \left( \norme{Z_0}_{L^{2}(\Omega)^n} + \norme{S}_{\mathcal{S}_2}\right),
\end{equation}
where $C_T = C e^{C/T}$. In particular, since $\rho_0$ is a continuous function satisfying $\rho_0(T)=0$, the above relation \eqref{EstiLTT} yields $Z(T,.)=0$.
\end{prop}
For the sake of completeness, the proof of \Cref{PropSourceTermL^2} is in \Cref{AppSourceTermMethod} (see \Cref{MethodLTTLinfty} applied with $r=2$).\\
\indent Now, we will deduce an observability estimate for the adjoint system:
\begin{equation}
\label{defsystZAdj}
\left\{
\begin{array}{l l}
- \partial_t \varphi - D_J^{\text{tr}} \Delta \varphi = A_J^{\text{tr}} \varphi&\mathrm{in}\ (0,T)\times\Omega,\\
\frac{\partial \varphi}{\partial \nu} = 0 &\mathrm{on}\ (0,T)\times\partial\Omega,\\
\varphi(T,.)=\varphi_{T} &\mathrm{in}\  \Omega.
\end{array}
\right.
\end{equation}
We have the following result which is an adaptation of \cite[Corollary 2.6]{LTT} or \cite[Theorem 4.1]{ImTak} (see \Cref{proofStrongObs} for a complete proof).
\begin{cor}\label{CorStrongObs}
There exists $C>0$ such that for every $\varphi_T \in L_{inv}^2$, the solution of \eqref{defsystZAdj} satisfies:
\begin{equation}
\label{ObsLTT}
\norme{\varphi(0,.)}_{L^2(\Omega)^n}^2 + \int_0^T \int_{\Omega} |\rho_{\mathcal{S}}(t) \varphi(t,x)|^2 dt dx \leq C_T \left( \sum\limits_{i=1}^m \int_0^T \int_{\omega} |\rho_{0}(t) \varphi_i(t,x)|^2 dt dx \right),
\end{equation}
where $C_T = C e^{C/T}$.
\end{cor}
\begin{rmk}
The estimate \eqref{ObsLTT} looks like a global Carleman inequality for \eqref{defsystZAdj}. But the strategy to get this type of estimate comes from the null-controllability theorem in $L_{inv}^2$ for \eqref{defsystZ} with an estimate of the cost and the source term method: \Cref{theolinear} and \Cref{PropSourceTermL^2}. We insist on the fact that we do not know how to prove the null-controllability in $L_{inv}^2$ of \eqref{defsystZ} by the usual global Carleman estimates applied to each equation of \eqref{defsystZAdj} when $m < n-4$ (see \Cref{comments} and in particular \Cref{OpenProb}).
\end{rmk}
In the next section, we take advantage of the strong observability estimate \eqref{ObsLTT} to get more regularity in $L^p$-sense for the control $H^J$.
\section{Construct $L^{\infty}$-controls and the source term method in $L^{\infty}$}
\subsection{Construct $L^{\infty}$-controls: the Penalized Hilbert Uniqueness Method}
The goal of this section is to prove a null-controllability result in $L^{\infty}$ with an estimate of the cost of the control.
\begin{theo}
\label{theolinearLinfty}
There exists $C>0$ such that for every $T>0$, $Z_0 \in L_{inv}^2$, there exists a control $H^J \in L^{\infty}((0,T)\times \Omega)^m$ verifying
\begin{equation}
\label{esticontrolThLinfty}
\norme{H^J}_{L^{\infty}(Q_T)^m} \leq C_T \norme{Z_0}_{L^2(\Omega)^n},\ \text{where}\ C_T =C e^{C/T}.
\end{equation}
and such that the solution $Z$ of \eqref{defsystZ} (see \Cref{SectionLinearization}) satisfies $Z(T,.)=0$.
\end{theo}
From now and until the end of the section, we will denote by $C_T$ various positive constants which can change from line to line and such that $C_T \leq C e^{C/T}$.\\

\indent In the next four parts, we perform the usual Penalized Hilbert Uniqueness Method, introduced for the first time by Viorel Barbu in \cite{B}. The idea is the following one: it is a well-known fact that the optimal control $H^J \in L^2((0,T)\times \Omega)^m$, i.e., the minimal-norm element in $L^2$, which steers the solution $Z$ of \eqref{defsystZ} to $0$ in time $T$ can be expressed as a function of a solution of the adjoint system \eqref{defsystZAdj} (see \cite[Section 1.4]{C} for more details in the context of linear finite dimensional controlled systems). By using the strong observability inequality \eqref{ObsLTT}, we will use this link by considering a penalized problem in $\mathcal{H}_2 \subset L^2((0,T)\times \Omega)^m$: the behavior at time $t=T$ of the weight $\rho_0$ will be the key point to produce more regular controls in $L^p$-sense.
\subsubsection{The beginning of the Penalized Hilbert Uniqueness Method}
Let us fix $Z_0 \in L_{inv}^2$.\\
\indent We define $P_{\varepsilon} : \mathcal{H}_2 \rightarrow \R^{+}$, by, for every $H^J \in \mathcal{H}_2$,
\begin{equation}
P_{\varepsilon}(H^J) := \frac{1}{2} \int\int_{(0,T)\times\omega} \rho_0^{-2}(t) |H^J(t,x)|^2dxdt + \frac{1}{2\varepsilon} \norme{Z(T,.)}_{L^2(\Omega)^n}^2,
\end{equation}
where $Z$ is the solution to the Cauchy problem \eqref{defsystZ} (see \Cref{SectionLinearization}) associated to the control $H^J$.\\
\indent 
The functional $P_{\varepsilon}$ is a $C^1$, coercive, strictly convex functional on the Hilbert space $\mathcal{H}_2$, then $P_{\varepsilon}$ has a unique minimum $H^{J,\varepsilon} \in \mathcal{H}_2$. Let $Z^{\varepsilon}$ be the solution to the Cauchy problem \eqref{defsystZ} with control $H^{J,\varepsilon}$ and initial data $Z_0$.\\
\indent The Euler-Lagrange equation gives
\begin{equation}
\forall H^J \in \mathcal{H}_2,\ \int\int_{(0,T)\times\omega} \rho_0^{-2}H^{J,\varepsilon} . H^J dxdt + \frac{1}{\varepsilon} \int_{\Omega}Z^{\varepsilon}(T,x).Z(T,x)dx = 0,
\label{eulerla}
\end{equation}
where $Z$ is the solution to the Cauchy problem \eqref{defsystZ} associated to the control $H^J$ and initial data $Z_0 = 0$. \\
\indent We introduce $\varphi^{\varepsilon}$ the solution to the adjoint problem \eqref{defsystZAdj} with final condition $\varphi^{\varepsilon}(T,.) = -\frac{1}{\varepsilon}Z^{\varepsilon}(T,.)$. A duality argument between $Z$ and $\varphi^{\varepsilon}$ gives
\begin{equation}
-\frac{1}{\varepsilon} \int_{\Omega} Z(T,x).Z^{\varepsilon}(T,x)dx =\int_{\Omega} Z(T,x).\varphi^{\varepsilon}(T,x)dx 
= \int\int_{(0,T)\times\omega} H^J .\varphi^{\varepsilon}.
\label{dual1}
\end{equation}
Then, we deduce from \eqref{eulerla} and \eqref{dual1} that 
\[ \forall H^J \in \mathcal{H}_2,\ \int\int_{(0,T)\times\omega}  \rho_0^{-2} H^{J,\varepsilon} . H^J  = \int\int_{(0,T)\times\omega} \varphi^{\varepsilon}.H^J .\]
Consequently, we have
\begin{equation}
\forall i \in \{1,\dots, m\},\ h_i^{\varepsilon} = \rho_0^{2} \varphi_i^{\varepsilon} 1_{\omega}.
\label{control-adj}
\end{equation}
Another duality argument applied between $Z^{\varepsilon}$ and $\varphi^{\varepsilon}$ together with \eqref{control-adj} gives
\begin{align}
- \frac{1}{\varepsilon} \int_{\Omega}|Z^{\varepsilon}(T,x)|^2dx &=\int_{\Omega} Z^{\varepsilon}(T,x).\varphi^{\varepsilon}(T,x)dx \notag\\
 &=  \int_{\Omega} Z_0(x).\varphi^{\varepsilon}(0,x)dx+ \int\int_{(0,T)\times\omega} H^{J,\varepsilon}.\varphi^{\varepsilon},\notag
 \end{align}
which yields
\begin{equation}
\label{dual2bis}
-\frac{1}{\varepsilon} \norme{Z^{\varepsilon}(T,.)}_{L^2(\Omega)^n}^2= \int_{\Omega} Z_0(x).\varphi^{\varepsilon}(0,x)dx + \sum\limits_{i=1}^m \int\int_{(0,T)\times\omega} |\rho_0 \varphi_i^{\varepsilon}|^{2}.
\end{equation}
By Young's inequality and the observability estimate \eqref{ObsLTT} applied to $\varphi^{\varepsilon}$, for $\delta >0$, we have:
\begin{align}
\label{YoungPHum}
&\left|\int_{\Omega} Z_0(x).\varphi^{\varepsilon}(0,x)dx\right| \\
&\leq \delta \norme{\varphi^{\varepsilon}(0,.)}_{L^2(\Omega)^n}^2  + C_{\delta} \norme{Z_0}_{L^2(\Omega)^n}\notag\\
& \leq \delta C_T \left(\sum\limits_{i=1}^m \int\int_{(0,T)\times\omega} |\rho_0(t) \varphi_i^{\varepsilon}(t,x)|^{2} dx dt\right) + C_{\delta} \norme{Z_0}_{L^2(\Omega)^n} .\notag
\end{align}
Then, by using \eqref{control-adj}, \eqref{dual2bis}, \eqref{YoungPHum} and by taking $\delta$ sufficiently small, we get
\begin{equation}
\frac{1}{\varepsilon} \norme{Z^{\varepsilon}(T,.)}_{L^2(\Omega)^n}^2 + \frac{1}{2}\norme{\rho_0^{-1} H^{J,\varepsilon}}_{L^2((0,T)\times\omega)^n}^2 \leq C_T \norme{Z_0}_{L^2(\Omega)^n}^2.
\label{inhum}
\end{equation}
\begin{rmk}
The estimate \eqref{inhum} yields \Cref{PropSourceTermL^2} for $S=0$ by letting $\varepsilon \rightarrow 0$. We remark that we have only used the term $\norme{\varphi(0,.)}_{L^2(\Omega)^n}^2$ in the left hand side of \eqref{ObsLTT}. The second term in the left hand side of \eqref{ObsLTT} enables to get more regularity (in $L^p$-sense) for the control $H^J$ (see \Cref{BootMeth} below).
\end{rmk}

\subsubsection{Maximal regularity theorems and Sobolev embeddings}
In this part, we recall a maximal regularity theorem in $L^p$ ($1 < p < +\infty$) for parabolic systems and an embedding result for Sobolev spaces.\\
\indent We introduce the following spaces: for every $r\in [1,+\infty]$,
\begin{equation*}
\label{defspaceslr}
W_{Ne}^{2,r}(\Omega) :=\left\{u \in W^{2,r}(\Omega)\ ;\ \frac{\partial u}{\partial \nu} = 0 \right\},\ X_r := L^r(0,T;W_{Ne}^{2,r}(\Omega))\cap W^{1,r}(0,T;L^{r}(\Omega)).
\end{equation*}
We have the following maximal regularity theorem.
\begin{prop}\cite[Theorem 2.1]{DHP}\label{wplp}\\
Let $1 < r < +\infty$, $k \in \N^*$, $D\in \mathcal{M}_k(\R)$ such that $Sp(D) \subset (0,+\infty)$, $A \in \mathcal{M}_{k}(\R)$ and $S \in L^r(Q_T)^k$. The following Cauchy problem admits a unique solution $U \in X_r^k $
\[
\left\{
\begin{array}{l l}
\partial_t U - D \Delta U= A U + S(t,x)&\mathrm{in}\ (0,T)\times\Omega,\\
\frac{\partial U}{\partial \nu} = 0 &\mathrm{on}\ (0,T)\times\partial\Omega,\\
U(0,.)=0 &\mathrm{in}\  \Omega.
\end{array}
\right.
\]
Moreover, there exists $ C>0$ independent of $S$ such that
\[ \norme{U}_{X_r^k} \leq C \norme{S}_{L^r(Q_T)^k}.\]
\end{prop}
We have the following embedding result for Sobolev spaces.
\begin{prop}\label{injsobo}\cite[Theorem 1.4.1]{WYW}\\
Let $r \in [1,+\infty[$, we have
\[ 
X_r \hookrightarrow 
\left\{
\begin{array}{c l}
L^{\frac{(N+2)r}{N+2-2r}}(Q_T) & \mathrm{if}\  r < \frac{N+2}{2},\\
L^{2 r}(Q_T) & \mathrm{if}\  r = \frac{N+2}{2},\\
L^{\infty}(Q_T) & \mathrm{if}\  r > \frac{N+2}{2}.
\end{array}
\right.
\]
\end{prop}

\subsubsection{Bootstrap method}\label{BootMeth}
In the next two parts, we will use the key identity between the control $H^{J, \varepsilon}$ and the solution of the adjoint system $\varphi^{\varepsilon}$, i.e, \eqref{control-adj} in order to deduce $L^p$-regularity for $H^{J, \varepsilon}$ from $L^p$-regularity for $\varphi^{\varepsilon}$. This kind of regularity will come from 
the application of successive $L^p$-parabolic regularity theorems stated in \Cref{wplp} to a modification of $\varphi^{\varepsilon}$ called $\psi^{\varepsilon,r}$ (see a precise definition in \eqref{Defpsir} below) which is bounded from below by $\rho_0^2 \varphi$. The beginning of this bootstrap argument is the strong observability inequality \eqref{ObsLTT}. Finally, we will pass to the limit ($\varepsilon \rightarrow 0$) in $\frac{1}{\varepsilon} \norme{Z^{\varepsilon}(T,.)}_{L^2(\Omega)^n}^2 \leq C_T \norme{Z_0}_{L^2(\Omega)^n}^2$ coming from \eqref{inhum} and $ \norme{H^{J,\varepsilon}}_{L^{\infty}(Q_T)}
 \leq C _T\norme{Z_0}_{L^2(\Omega)^n}$ coming from \eqref{estinftyhaux} (see below). \\
\indent By using \Cref{rmkp>q^2}, we introduce the positive real number
\begin{equation}
\label{deflambda}
\gamma := 2p - (1+p)q^{2} > 0.
\end{equation}
Let us define a sequence of increasing positive real numbers $(\gamma_r)_{r \in \N}$ such that $\lim\limits_{r \rightarrow + \infty} \gamma_r =\gamma$, where $\gamma$ is defined in \eqref{deflambda}.\\
\indent We introduce for every $r \in \N$,
\begin{equation}
\label{defrhoGr}
\rho_{\mathcal{S},r}(t) = M^{-1-p} \exp\left(-\frac{\left((1+p)q^{2} +\gamma_r\right)M}{(q-1)(T-t)}\right).
\end{equation}
Then, we have from \eqref{defrho0bis}, for every $r \in \N$,
\begin{equation}
\label{Comprhorhor}
\rho_0^{2} \leq C_T \rho_{\mathcal{S},r}.
\end{equation}
We remark that we have for every $r \in \N$,
\begin{equation}
\label{rhoGrr+1}
|\rho_{\mathcal{S},r+1}'(t)| \leq C_{T,r} \rho_{\mathcal{S},r}(t).
\end{equation}

\indent We define for every $r \in \N$,
\begin{equation}
\psi^{\varepsilon,r}(t,x):=\rho_{\mathcal{S},r}(t)\varphi^{\varepsilon}(t,x).
\label{Defpsir}
\end{equation}
From \eqref{defsystZAdj}, \eqref{defrhoGr} and \eqref{Defpsir}, we have for every $r \in \N^{*}$,
\begin{equation}
\left\{
\begin{array}{l l}
-\partial_t{\psi^{\varepsilon,r}} - D_J^{\text{tr}}\Delta\psi^{\varepsilon,r}  = A_J^{\text{tr}}\psi^{\varepsilon,r}  - \rho_{\mathcal{S},r}'(t) \varphi^{\varepsilon} &\mathrm{in}\ (0,T)\times\Omega,\\
\frac{\partial\psi^{\varepsilon,r}}{\partial \nu} = 0 &\mathrm{on}\ (0,T)\times\partial\Omega,\\
\psi^{\varepsilon,r}(T,.)=0 &\mathrm{in}\ \Omega.
\end{array}
\right.
\label{systpsi}
\end{equation}
By using \eqref{rhoGrr+1}, we remark that
\begin{equation}
| - \rho_{\mathcal{S},r}'(t) \varphi^{\varepsilon}| \leq C_T|\psi^{\varepsilon,r-1}|.
\label{insource}
\end{equation}
\indent Let $(p_r)_{r \in \N}$ be the following sequence defined by induction
\begin{equation}
\label{defp0}
p_0 = 2, 
\end{equation}
\begin{equation}
\label{defpr}
p_{r+1} :=
\left\{
\begin{array}{c l}
\frac{(N+2)p_{r}}{N+2-2p_{r}} & \mathrm{if}\  p_{r} < \frac{N+2}{2},\\
2 p_{r} & \mathrm{if}\  p_{r} = \frac{N+2}{2},\\
+\infty & \mathrm{if}\  p_{r} > \frac{N+2}{2}.\\
\end{array}
\right.
\end{equation}
There exists $l\in \N^{*}$ such that 
\begin{equation}
\forall r \geq l,\ p_r = + \infty.
\label{pinfty}
\end{equation}
We show, by induction, that for every $0 \leq r  \leq l$, we have
\begin{equation}
\psi^{\varepsilon,r} \in L^{p_r}(Q_T)^n\ \mathrm{and} \norme{\psi^{\varepsilon,r}}_{L^{p_r}(Q_T)^n} \leq C_T \norme{Z_0}_{L^2(\Omega)^n}.
\label{estsolu}
\end{equation}
\indent The case $r=0$ can be deduced from the fact that $\gamma_0 > 0$ and the observability estimate \eqref{ObsLTT} ($p_0 =2$ by \eqref{defp0}).\\
\indent Let $r \in \N^{*}$. We assume that 
\begin{equation}
\psi^{\varepsilon,r-1} \in L^{p_{r-1}}(Q_T)^n\ \mathrm{and} \norme{\psi^{\varepsilon,r-1}}_{L^{p_{r-1}}(Q_T)^n} \leq C_T \norme{Z_0}_{L^2(\Omega)^n}.
\label{hr}
\end{equation}
Then, from \eqref{systpsi}, \eqref{insource}, \eqref{hr} and from the maximal regularity theorem: \Cref{wplp} applied with $p_{r-1} \in (1,+\infty)$, we get
\begin{equation}
\psi^{\varepsilon,r} \in X_{p_{r-1}}^n\ \mathrm{and} \norme{\psi^{\varepsilon,r}}_{X_{p_{r-1}}^r} \leq C_T \norme{Z_0}_{L^2(\Omega)^n}.
\label{regpar}
\end{equation}
Moreover, by the Sobolev embedding: \Cref{injsobo} and \eqref{defpr}, we have
\begin{equation*}
\psi^{\varepsilon,r} \in L^{p_r}(Q_T)^n\ \mathrm{and} \norme{\psi^{\varepsilon,r}}_{L^{p_r}(Q_T)^n} \leq C_T \norme{Z_0}_{L^2(\Omega)^n}.
\end{equation*}
This concludes the induction.
\subsubsection{The end of the Penalized Hilbert Uniqueness Method}
Now, by applying consecutively \eqref{pinfty} ($p_l = +\infty$), \eqref{control-adj}, \eqref{Comprhorhor} and \eqref{estsolu}, we have for every $i \in \{1,\dots,m\}$,
\begin{equation}\small
\norme{h_i^{\varepsilon}}_{L^{\infty}(Q_T)}
=\norme{\rho_0^{2}\varphi_i^{\varepsilon}}_{L^{p_l}(Q_T)}
\leq C_T \norme{\rho_{S,l}\varphi_i^{\varepsilon}}_{L^{p_l}(Q_T)}
\leq C_T \norme{\psi_i^{\varepsilon,l}}_{L^{p_l}(Q_T)}
 \leq C _T\norme{Z_0}_{L^2(\Omega)^n}.
\label{estinftyhaux}
\end{equation}
\normalsize
Therefore, from \eqref{estinftyhaux}, $(H^{J,\varepsilon})_{\varepsilon}$ is bounded in $L^{\infty}(Q_T)^m$, then up to a subsequence, we can assume that there exists $H^J \in L^{\infty}(Q_T)^m$ such that
\begin{equation}
H^{J,\varepsilon} \underset{\varepsilon \rightarrow 0}{\rightharpoonup}^* H^J\ \text{in}\ L^{\infty}(Q_T)^m,
\label{convhjweakst}
\end{equation} 
and
\begin{equation}
\norme{H^{J}}_{L^{\infty}(Q_T)^m}\leq C_T \norme{Z_0}_{L^2(\Omega)^n}.
\label{BoundH^JLim}
\end{equation}
From \eqref{estinftyhaux}, \Cref{wpl2linfty} applied to \eqref{defsystZ} satisfied by $Z^{\varepsilon}$, we obtain
\begin{equation}
\norme{Z^{\varepsilon}}_{W_T^n} \leq C_T \norme{Z_0}_{L^2(\Omega)^n}.
\label{uneborneenY}
\end{equation}
So, from \eqref{uneborneenY}, up to a subsequence, we can suppose that there exists $Z \in W_T^n$ such that
\begin{equation}
Z^{\varepsilon} \underset{\varepsilon \rightarrow 0}{\rightharpoonup} Z\ \text{in}\ L^2(0,T;H^1(\Omega)^n),\ \partial_t Z^{\varepsilon} \underset{\varepsilon \rightarrow 0}{\rightharpoonup} \partial_tZ\ \text{in}\ L^2(0,T;(H^1(\Omega))'^n),
\label{wkVeps}
\end{equation}
and from \eqref{injclassique},
\begin{equation}
Z^{\varepsilon}(0,.) \underset{\varepsilon \rightarrow 0}{\rightharpoonup} Z(0,.)\ \text{in}\ L^2(\Omega)^n,\ Z^{\varepsilon}(T,.) \underset{\varepsilon \rightarrow 0}{\rightharpoonup} Z(T,.)\ \text{in}\  L^2(\Omega)^n.
\end{equation}
Then, as we have $Z^{\varepsilon}(0,.)=Z_0$ and $Z^{\varepsilon}(T,.) \rightarrow 0$ from \eqref{inhum}, we deduce that 
\begin{equation}
Z(0,.) = Z_0,\  \mathrm{and}\  Z(T,.) = 0.
\label{cicf}
\end{equation}
By letting $\varepsilon \rightarrow 0$, we have from \eqref{wkVeps}, \eqref{convhjweakst} and \eqref{cicf} that
\begin{equation}
\label{defsystZLimite}
\left\{
\begin{array}{l l}
 \partial_t Z - D_J \Delta Z = A_J Z  + H^{J} 1_{\omega}&\mathrm{in}\ (0,T)\times\Omega,\\
\frac{\partial Z}{\partial \nu} = 0 &\mathrm{on}\ (0,T)\times\partial\Omega,\\
(Z(0,.),Z(T,.))=(Z_{0},0) &\mathrm{in}\  \Omega.
\end{array}
\right.
\end{equation}
This ends the proof of \Cref{theolinearLinfty} by using \eqref{BoundH^JLim} and \eqref{defsystZLimite}. 
\subsection{The come back to the source term method in $L^{\infty}$}
The goal of this section is to apply the source term method in $L^{\infty}$ thanks to the null-controllability result in $L^{\infty}$: \Cref{theolinearLinfty}.\\
\indent To simplify the notations, we assume that the control cost in $L^{\infty}$ of \Cref{theolinearLinfty} satisfies: $C_T \leq M e^{M/T}$ where $M$ is already defined at the beginning of \Cref{SourcetermL^2}.\\
\indent From \Cref{MethodLTTLinfty} with $r=+\infty$ proved in \Cref{AppSourceTermMethod}, we deduce the following null-controllability result for \eqref{defsystZSourceS} (see \Cref{SourcetermL^2}) in $L^{\infty}$.
\begin{prop}\label{ControlPropEstiLTTLinfty}
For every $S \in \mathcal{S}_{\infty}$ and $Z_0 \in L_{inv}^{\infty}$, there exists $H^J \in \mathcal{H}_{\infty}$, such that the solution $Z$ of \eqref{defsystZSourceS} satisfies $Z \in \mathcal{Z}_{\infty}$. Furthermore, there exists $C >0$, not depending on $S$ and $Z_0$, such that
\begin{equation}
\label{EstiLTTLinfty}
\norme{Z/\rho_0}_{L^{\infty}([0,T];L^{\infty}(\Omega)^n)} + \norme{H^J}_{\mathcal{H}_{\infty}} \leq C \left( \norme{Z_0}_{L^{\infty}(\Omega)^n} + \norme{S}_{\mathcal{S}_{\infty}}\right).
\end{equation}
In particular, since $\rho_0$ is a continuous function satisfying $\rho_0(T)=0$, the above relation \eqref{EstiLTTLinfty} yields $Z(T,.)=0$.
\end{prop}

\section{The inverse mapping theorem in appropriate spaces}
\label{Fixed point}

The goal of this section is to prove \Cref{LocContrLinfty}. The proof is based on \Cref{ControlPropEstiLTTLinfty} and an inverse mapping theorem in suitable spaces.
\begin{proof}
Let us introduce the following space (see the definitions \eqref{defpoidsg}, \eqref{defpoidsf} and \eqref{defpoidsu}):
\begin{equation}
E := \{(Z,H^J) \in \mathcal{Z}_{\infty} \times \mathcal{H}_{\infty};\ \partial_t Z - D_J \Delta Z - A_J Z - H^J 1_{\omega} \in S_{\infty}\}.
\label{DefBanachE}
\end{equation}
We endow $E$ with the following norm: for every $(Z, H^J) \in E$, 
\begin{equation}
\norme{(Z,H^J)}_{E} = \norme{Z(0,.)}_{L^{\infty}} + \norme{Z}_{\mathcal{Z}_{\infty}} + \norme{H^J}_{\mathcal{H}_{\infty}} + \norme{\partial_t Z - D_J \Delta Z - A_J Z - H^J 1_{\omega}}_{S_{\infty}}.
\label{normeE}
\end{equation}
Then, $(E, \norme{.}_{E})$ is a Banach space.\\
\indent For every $Z \in \mathcal{Z}_{\infty}$, we introduce the following polynomial nonlinearity of degree more than $2$:
\begin{equation}
\label{Nonlinearity}
Q(Z) := G(Z) - A_J Z,
\end{equation}
where $G$ is defined in \eqref{defGgi}.  By denoting $\gamma := \max\left(\sum\limits_{i=1}^n \alpha_i, \sum\limits_{i=1}^n \beta_i\right)$, we remark that for every $Z \in \mathcal{Z}_{\infty}$, $Q(Z) = \sum\limits_{i=2}^{\gamma} Q_i(Z)$ where for every $2 \leq i \leq \gamma$, $Q_i(Z)$ is a polynomial term with respect to $Z=(z_1, \dots, z_n)$ of degree $i$. By using \eqref{rhorhoG}, we deduce that $Q(Z) \in \mathcal{S}_{\infty}$ and for every $2 \leq i \leq \gamma$,
\begin{equation}
\norme{Q_i(Z)}_{\mathcal{S}_{\infty}} = \norme{\frac{Q_i(Z)}{\rho_{\mathcal{S}}}}_{L^{\infty}(Q_T)^n} = \norme{\rho_{0}^{i-2}\frac{\rho_{0}^2}{\rho_{\mathcal{S}}} \frac{Q_i(Z)}{\rho_{0}^i}}_{L^{\infty}(Q_T)^n} \leq C \norme{Z}_{\mathcal{Z}_{\infty}}^i.
\label{MajQ}
\end{equation}
\indent We introduce the following mapping:
\begin{equation}
\begin{array}{lll}
\mathcal{A} :& E &\longrightarrow F:=S_{\infty} \times L_{inv}^{\infty} \\
     &(Z,H) &\longmapsto (\partial_t Z - D_J \Delta Z - A_J Z - H^J 1_{\omega} - Q(Z), Z(0,.)).
\end{array}
\label{DefMapA}
\end{equation}
By using \eqref{DefBanachE}, the fact that for $(Z,H^J) \in E$ and $Q(Z) \in S_{\infty}$ by \eqref{MajQ}, we see that $\mathcal{A}$ is well-defined. Moreover, $\mathcal{A} \in C^1(E;F)$. Indeed, all the terms in \eqref{DefMapA} are linear and continuous (thus $C^{\infty}$) thanks to \eqref{normeE} except the term $Q(Z)$. And, for $(Z,H^J) \in E$, $Q(Z)$ is a polynomial function with respect to $Z$ which is $C^{\infty}$. Indeed,
$$Q_i(Z) = \sum\limits_{\substack{\gamma_1+\dots +\gamma_n=i\\ \gamma_1, \dots, \gamma_n \geq 0}} c_{\gamma_1, \dots, \gamma_n} z_1^{\gamma_1} \dots z_n^{\gamma_n} =  \sum\limits_{\substack{\gamma_1+\dots +\gamma_n=i\\ \gamma_1, \dots, \gamma_n \geq 0}} c_{\gamma_1, \dots, \gamma_n} \mathcal{B}_{\gamma_1, \dots, \gamma_n} \circ \mathcal{L}_{\gamma_1, \dots, \gamma_n}(Z),$$
where
$$ \mathcal{L}_{\gamma_1, \dots, \gamma_n}(z_1, \dots, z_n) := (\underbrace{z_1, \dots, z_1}_{\gamma_1\ \text{times}}, \underbrace{z_2, \dots, z_2}_{\gamma_2\ \text{times}}, \dots, \underbrace{z_n, \dots, z_n}_{\gamma_n\ \text{times}}),$$
$$ \mathcal{B}_{\gamma_1, \dots, \gamma_n}(\underbrace{a_{1, 1}, \dots, a_{1, \gamma_1}, a_{2,1}, \dots, a_{2, \gamma_2}, \dots, a_{n,1},  \dots, a_{n,\gamma_n}}_{a_{\gamma_1, \dots, \gamma_n}}):=\prod\limits_{i=1}^n \prod\limits_{j=1}^{\gamma_i} a_{i,j}.$$
To simplify, we renote $\mathcal{L}:=\mathcal{L}_{\gamma_1, \dots, \gamma_n}$ and $\mathcal{B}:=\mathcal{B}_{\gamma_1, \dots, \gamma_n}$. The mapping $\mathcal{L}$ is $C^{\infty}$ because $\mathcal{L}$ is linear and continuous. The mapping $\mathcal{B}$ is $C^{\infty}$ because $\mathcal{B}$ is $i$-linear and continuous. Indeed, by using \eqref{rhorhoG}, we have
\begin{align*}
\norme{ \mathcal{B}(a_{\gamma_1, \dots, \gamma_n})}_{\mathcal{S}_{\infty}} &= \norme{\frac{\mathcal{B}(a_{\gamma_1, \dots, \gamma_n})}{\rho_{\mathcal{S}}}}_{L^{\infty}(Q_T)^n} = \norme{\rho_{0}^{i-2}\frac{\rho_{0}^2}{\rho_{\mathcal{S}}} \frac{\mathcal{B}(a_{\gamma_1, \dots, \gamma_n})}{\rho_{0}^i}}_{L^{\infty}(Q_T)^n}\\
& \leq C \prod\limits_{i=1}^n \prod\limits_{j=1}^{\gamma_i}\norme{a_{i,j}}_{\mathcal{Z}_{\infty}}.
\end{align*}
Moreover, the differential of $\mathcal{A}$ at the point $(0,0)$ in the direction $(Z,H^J)$ is
\begin{equation}
D \mathcal{A} (0,0) . (Z,H^J) = (\partial_t Z - D_J \Delta Z - A_J Z - H^J 1_{\omega}, Z(0,.)),
\label{DerivA}
\end{equation}
which is onto by using \Cref{ControlPropEstiLTTLinfty}. Then, by using the inverse mapping theorem (see \cite[Theorem 2]{CGMR}), we deduce that there exists $r > 0$, such that for every $(S, Z_0) \in F$ satisfying $\norme{(S,Z_0)}_{F} \leq r$, there exists $(Z,H^J) \in E$ such that $\mathcal{A}(Z, H^J) = (S,Z_0)$. By taking $S=0$ and $Z_0 \in L_{inv}^{\infty}$ such that $\norme{Z_0}_{L^{\infty}(\Omega)^n} \leq r$, we get the existence of $(Z,H^J) \in \mathcal{Z}_{\infty} \times \mathcal{H}_{\infty}$ such that 
\begin{equation*}
\label{defsystZNLProof}
\left\{
\begin{array}{l l}
 \partial_t Z - D_J \Delta Z = \underbrace{A_J Z + Q(Z)}_{G(Z)\ \text{by}\ \eqref{Nonlinearity}} + H^{J} 1_{\omega}&\mathrm{in}\ (0,T)\times\Omega,\\
\frac{\partial Z}{\partial \nu} = 0 &\mathrm{on}\ (0,T)\times\partial\Omega,\\
(Z(0,.),Z(T,.))=(Z_{0},0) &\mathrm{in}\  \Omega.
\end{array}
\right.
\end{equation*}
This concludes the proof of \Cref{LocContrLinfty}.
\end{proof}

\section{Comments}
\label{comments}

\subsection{More general semilinearities}

In this paper, we have only considered particular semilinearities of the form:
\begin{equation}
\forall 1 \leq i \leq n, \ f_i(u_1, \dots, u_n) =  (\beta_i-\alpha_i)\left(\prod\limits_{k=1}^n u_k^{\alpha_k} - \prod\limits_{k=1}^n u_k^{\beta_k}\right).
\label{defiPol}
\end{equation}
But the main result of the article, i.e., \Cref{mainresult1} holds true with more general polynomial semilinearities satisfying
$$\exists R \in \R[X_1, \dots, X_n],\ \forall 1 \leq i \leq n,\ \exists a_i \in \R^{*},\ f_i = a_i R,$$
where $\R[X_1, \dots, X_n]$ denotes the space of multivariate polynomials with coefficients in $\R$. In this case,  $(u_i^{*})_{1 \leq i \leq n}$ is a constant nonnegative stationary state if 
$$(u_i^{*})_{1 \leq i \leq n} \in [0,+\infty)^{n}\ \text{and}\ R(u_1^{*}, \dots, u_n^{*}) = 0.$$
For example, \eqref{defiPol} rewrites as follows
$$ \forall 1 \leq i \leq n, \ f_i(X_1, \dots, X_n) =  (\beta_i-\alpha_i)\left(\prod\limits_{k=1}^n X_k^{\alpha_k} - \prod\limits_{k=1}^n X_k^{\beta_k}\right).$$

\subsection{Degenerate cases}
\label{SectionDegen}
In this part, we assume that \Cref{asscouplordre0} is not satisfied. Then, the usual strategy is to perform the return method, introduced by Jean-Michel Coron in \cite{C2} (see also \cite[Chapter 6]{C}). This method consists in finding a reference trajectory $(\overline{U}, \overline{H^J})$ verifying $U(0,.)=U(T,.)=U^{*}$ of \eqref{systNLC} (see \Cref{SectionControlOpQuestion}) such that the linearized system of \eqref{systNLC} around $(\overline{U}, \overline{H^J})$ is null-controllable.

\begin{Exam}
For $n=2$, we take $\alpha_1 = 3$, $\beta_1 = 0$, $\alpha_2 = 0$, $\beta_2=1$ and $J=\{1\}$. We get the following control reaction-diffusion system:
\begin{equation}\label{systcube}
\left\{
\begin{array}{l l}
\partial_t u_1 - d_1 \Delta u_{1} = -u_1^{3} +u_2  + h_1 1_{\omega}  &\mathrm{in}\ (0,T)\times\Omega,\\
\partial_t u_2 - d_2 \Delta u_{2} = u_1^{3} - u_2   &\mathrm{in}\ (0,T)\times\Omega,\\
\frac{\partial U}{\partial \nu} = 0 &\mathrm{on}\ (0,T)\times\partial\Omega,\\
U(0,.)=U_0 &\mathrm{in}\  \Omega.
\end{array}
\right.
\end{equation}
In this case, \Cref{asscouplordre0} is not satisfied if and only if $(u_1^{*}, u_2^{*})=(0,0)$. By using the return method, Jean-Michel Coron, Sergio Guerrero and Lionel Rosier prove the local null-controllability around $(0,0)$ of \eqref{systcube} (see \cite{CGR}).
\end{Exam}

\begin{Exam}
\label{RMArt1}
For $n=4$, we take $\alpha_1 = \alpha_3 = \beta_2 = \beta_4 = 1$, $\alpha_2 = \alpha_4 = \beta_1= \beta_3 = 0$ and $J=\{1,2,3\}$. Then, we get the controlled reaction-diffusion system as in \eqref{systNLCu1u3u2u4}. In this case, \Cref{asscouplordre0} is not satisfied if and only if $(u_1^{*}, u_3^{*}, u_4^{*})=(0,0,0)$. More precisely, the linearized system around $\Big((0,u_2^*, 0,0),(0,0,0)\Big)$ is not null-controllable because the fourth equation is decoupled from the others:
$$\partial_t u_4 - d_4 \Delta u_{4} = u_2^* u_4.$$
By using the return method, the author proves the local null-controllability around $(0,u_2^{*},0,0)$ of \eqref{systNLCu1u3u2u4} (see \cite{LB}). More precisely, for this system, a reference trajectory is not difficult to construct. Indeed, one can take $\big((0,u_2^*,g,0),(0,0,\partial_t g - d_3 \Delta g)\big)$ where $g$ satisfies
\begin{equation}
g \in C^{\infty}(\overline{Q}),\ g\geq 0,\ g \neq 0,\ supp(g) \subset (0,T)\times \omega. 
\label{defgLoc}
\end{equation}
Thus, the fourth equation of the linearized system around this trajectory is:
$$\partial_t u_4 - d_4 \Delta u_{4} = g(t,x) u_1  - u_2^* u_4.$$
Roughly speaking, the component $u_4$ can be controlled throughout the coupling term $g(t,x) u_1$ which lives in the control zone (see for instance \cite[Section 3.1, Lemma 3]{CGR}).
\end{Exam}

\begin{Exam}
\label{app3}
For $n=10$, we take $\alpha_1 = \alpha_3 = \alpha_5 = \alpha_7 = \alpha_9 = \beta_2 = \beta_4 = \beta_6 = \beta_8 = \beta_{10} = 1$, $\alpha_2 = \alpha_4 = \alpha_6 = \alpha_8 = \alpha_{10} = \beta_1 = \beta_3 = \beta_5 = \beta_7 = \beta_9 = 0$ and $J=\{1,2,3,4,5\}$. The control system is the following one:
\small
\begin{equation}\label{systNLCu1u3u5u7u9u2u4u6u8u10}
\forall 1 \leq i \leq 10,\ \left\{
\begin{array}{l l}
\partial_t u_i - d_i \Delta u_{i} = \\(-1)^{i}(u_1 u_3 u_5u_7u_9 - u_2 u_4u_6u_8u_{10)} + h_i 1_{\omega} 1_{i \in J} &\mathrm{in}\ (0,T)\times\Omega,\\
\frac{\partial u_i}{\partial \nu} = 0 &\mathrm{on}\ (0,T)\times\partial\Omega,\\
u_i(0,.)=u_{i,0} &\mathrm{in}\  \Omega.
\end{array}
\right.
\end{equation}
\normalsize
The stationary state
$$ (0,0,0,0,u_5^{*},u_6^{*}, u_7^{*}, u_8^{*}, u_9^{*}, u_{10}^{*}),$$
where $(u_5^{*},u_7^{*}, u_9^{*}) \in (0,+\infty)^{3}$ and $(u_6^{*},u_8^{*}, u_{10}^{*}) \in [0,+\infty)^{3}$ does not satisfy \Cref{asscouplordre0}. In this case, the linearized system around \\$\Big((0,0,0,0,u_5^{*},u_6^{*}, u_7^{*}, u_8^{*}, u_9^{*}, u_{10}^{*}),(0, 0,0,0, 0)\Big)$ is not null-controllable because the sixth equation of this system is:
$$\partial_t u_6 - d_6 \Delta u_6 = 0.$$
As for \Cref{RMArt1}, we can easily construct a reference trajectory:
$$\Big((0,0,g,0,u_5^{*},u_6^{*}, u_7^{*}, u_8^{*}, u_9^{*}, u_{10}^{*}),(0,0,\partial_t g - d_3 \Delta g,0,0)\Big),$$
where $g$ satisfies \eqref{defgLoc}. By performing the same change of variables as in \Cref{sectionChangeVar} and by linearizing around the reference trajectory, we find the same system as in \eqref{defsystZ} (see \Cref{SectionLinearization} with $n=10$, $m=5$) where the coefficients of $A$ are allowed to depend on $(t,x)$ and $a_{61}(t,x) \geq \varepsilon > 0$ on $(t_1,t_2) \times \omega_0 \subset (0,T) \times \omega$. This linear system seems to be null-controllable according to the following heuristic diagram:
\begin{align*}
&h_1 \xrightarrow[]{controls} z_1,\ h_2 \xrightarrow[]{controls} z_2,\ h_3 \xrightarrow[]{controls} z_3,\ h_4 \xrightarrow[]{controls} z_4,\ h_{5}\xrightarrow[]{controls} z_{5}, \\ 
&z_1  \xrightarrow[a_{61}(t,x) z_1]{controls}z_{6} \xrightarrow[{\Delta z_{6}}]{controls} z_{7}\xrightarrow[{\Delta z_{7}}]{controls} z_{8}\xrightarrow[{\Delta z_{8}}]{controls} z_9  \xrightarrow[{\Delta z_{9}}]{controls} z_{10}.
\end{align*}
Unfortunately, we do not know how to prove that the linearized system around this trajectory is null-controllable for technical reasons maybe. It comes from the fact that in this case $m=5 < n-4 = 6$. Intuitively, with the proof strategy performed in \cite{LB}, we have to benefit from one coupling term of order $0$ (in $L^{\infty}$) and four coupling terms of order $2$. This leads to the following open problem.
\end{Exam}
We introduce the following notation: $C_b^{\infty}(\overline{Q_T})$ is the set of functions $A$ defined on $\overline{Q_T}$ of class $C^{\infty}$ and such that all the derivatives of $A$ are bounded. We choose to state the following open problem for Dirichlet conditions instead of Neumann conditions to avoid the constraints on the initial data.
\begin{op}\label{OpenProb}
Let $n,m$ be two integers such that $n \geq 6$, $m < n-4$ and $(d_i)_{1 \leq i \leq n} \in (0,+\infty)^n$. Let $A \in C_b^{\infty}(\overline{Q_T})^{m \times n}$. We assume that there exist $(t_1,t_2) \subset (0,T)$, a nonempty open subset $\omega_0$ such that $\omega_0 \subset\subset \omega$ and $\varepsilon > 0$ such that $A_{m+1,m}(t,x) \geq \varepsilon$ on $(t_1,t_2)\times\omega_0$. For $y_0 \in L^2(\Omega)^n$, $(h_i)_{1 \leq i \leq m} \in L^2((0,T)\times\Omega)^m$, we consider the linear control system:
\small
\begin{align}
\left\{
\begin{array}{ll l}
\partial_t y_i - d_i \Delta y_{i} = \sum\limits_{j=1}^n A_{i,j}(t,x) y_j + h_i 1_{\omega} 1_{1 \leq i \leq m} &\mathrm{in}\ (0,T)\times\Omega,&1 \leq i \leq m+1\\
 \partial_t y_i - d_i \Delta y_{i} = \Delta y_{i-1} &\mathrm{in}\ (0,T)\times\Omega,& m+2 \leq i \leq n\\
y = 0 &\mathrm{on}\ (0,T)\times\partial\Omega,\\
y(0,.)=y_{0} &\mathrm{in}\  \Omega.
\end{array}
\right.
\label{systNLCY2}
\end{align}
\normalsize
Is the system \eqref{systNLCY2} null-controllable in $L^2(\Omega)^n$?
\end{op}
\begin{rmk}
\Cref{OpenProb} is closely related to the generalization of \cite[Theorem 1.1]{FCGBT} to linear parabolic systems with diffusion matrices that contain Jordan blocks of dimension more than $5$. Indeed, the diffusion matrix of \eqref{systNLCY2} is $D_J$ defined in \eqref{defNewD}. The submatrix $D_{\sharp}$ (see again \eqref{defNewD}) looks like a Jordan block of dimension more than $5$ if $m<n-4$. Consequently, the strategy of Carleman inequalities applied to each equation of the adjoint system of \eqref{systNLCY2} yields some global terms in the right hand side of the inequality that cannot be absorbed by the left hand side (see \cite[Section 2]{FCGBT}).
\end{rmk}

\appendix
\section{}\label{appendix}

\subsection{Stationary states}\label{StatioStates}

We only have considered nonnegative stationary constant solutions of \eqref{systNLF}. It is not restrictive because of the following proposition.
\begin{prop}\label{solutionspositivesconstantes}
Let $(u_i)_{1 \leq i \leq n} \in C^2(\overline{\Omega})^n$ be a nonnegative solution of
\begin{equation}
\forall 1 \leq i \leq n,\ 
\left\{
\begin{array}{l l}
-d_i \Delta u_{i} = f_i(U) &\mathrm{in}\ \Omega,\\
\frac{\partial u_i}{\partial \nu} = 0 &\mathrm{on}\ \partial\Omega,
\end{array}
\right.
\label{syststat}
\end{equation}
where $f_i(U)$ ($1 \leq i \leq n$) is defined in \eqref{deffi}. Then, for every $1 \leq i \leq n$, $u_i$ is constant.
\end{prop}
The proof relies on an entropy inequality: $-\sum_{i=1}^{n} \log(u_i) f_i(U)  \leq 0$.
\begin{proof}
Let $\varepsilon > 0$ be a small parameter. For every $1 \leq i \leq n$, we introduce 
$$u_{i,\varepsilon} = u_i +\varepsilon,\qquad w_{i,\varepsilon} = u_{i,\varepsilon}(\log u_{i,\varepsilon} - 1)+1 \geq 0.$$
We have 
\begin{equation}
\forall 1\leq i \leq n,\ \nabla w_{i,\varepsilon} = \log(u_{i,\varepsilon}) \nabla u_{i,\varepsilon},\qquad \Delta w_{i,\varepsilon} = \log(u_{i,\varepsilon}) \Delta u_{i,\varepsilon} + \frac{|\nabla u_{i,\varepsilon}|^2}{u_{i,\varepsilon}}.
\label{statsol2}
\end{equation}
Then, from \eqref{syststat} and \eqref{statsol2}, we have
\begin{equation}
\forall 1 \leq i \leq n,\ 
\left\{
\begin{array}{l l}
-d_i \Delta w_{i,\varepsilon} + d_i \frac{|\nabla u_{i,\varepsilon}|^2}{u_{i,\varepsilon}} = -\log(u_{i,\varepsilon}) f_i(U) &\mathrm{in}\ \Omega,\\
\frac{\partial w_{i,\varepsilon}}{\partial n} = 0 &\mathrm{on}\ \partial\Omega.
\end{array}
\right.
\label{syststatw}
\end{equation}
We sum the $n$ equations of \eqref{syststatw}, we integrate on $\Omega$ and we use the increasing of the function $\log$:
\begin{align}
\label{statsol3}
&0 + \int_{\Omega} \sum\limits_{i=1}^{n} d_i \frac{|\nabla u_{i,\varepsilon}|^2}{u_{i,\varepsilon}}\\
&= - \Bigg(\int_{\Omega} \left\{\log\left(\prod\limits_{i=1}^n u_{i,\varepsilon}^{\alpha_i}\right) - \log\left(\prod\limits_{i=1}^n u_{i,\varepsilon}^{\beta_i}\right)\right\}\left\{\prod\limits_{i=1}^n u_i^{\alpha_i} - \prod\limits_{i=1}^n u_i^{\beta_i}\right\}\Bigg)\notag\\
&=  - \Bigg(\int_{\Omega} \left\{\log\left(\prod\limits_{i=1}^n u_{i,\varepsilon}^{\alpha_i}\right) - \log\left(\prod\limits_{i=1}^n u_{i,\varepsilon}^{\beta_i}\right)\right\}\left\{\prod\limits_{i=1}^n u_{i,\varepsilon}^{\alpha_i} - \prod\limits_{i=1}^n u_{i,\varepsilon}^{\beta_i} + \mathcal{O}(\varepsilon) \right\}\Bigg)\notag\\
& \leq  \int_{\Omega} \left|\log\left(\prod\limits_{i=1}^n u_{i,\varepsilon}^{\alpha_i}\right) - \log\left(\prod\limits_{i=1}^n u_{i,\varepsilon}^{\beta_i}\right)\right| \mathcal{O}(\varepsilon) \leq \left(\sum\limits_{i=1}^n (\alpha_i + \beta_i) \right) |\log(\varepsilon)| \mathcal{O}(\varepsilon) \underset{\varepsilon \rightarrow 0} \rightarrow 0.\notag
\end{align}
Moreover,
\begin{equation}\forall 1 \leq i \leq n,\ \int_{\Omega}d_i \frac{|\nabla u_i^{\varepsilon}|^2}{u_i^{\varepsilon}} = \int_{\Omega}4 d_i|\nabla \sqrt{u_i^{\varepsilon}}|^2.\label{uiepsnonconstant}\end{equation} 
Consequently, from \eqref{statsol3}, \eqref{uiepsnonconstant}, we get that
$$ \forall 1 \leq i \leq n,\ \int_{\Omega}4 d_i|\nabla \sqrt{u_i}|^2 = 0.$$
Consequently, for every $1 \leq i \leq n$, $u_i$ is constant.
\end{proof}

Our proof of \Cref{mainresult1} does not treat the case of stationary states which can change of sign, contrary to the proof of \cite[Theorem 3.2]{LB} (see \cite[Section 6.2]{LB}). As in the previous part (see \Cref{app3}), the proof of \cite[Theorem 3.2]{LB} can be adapted to local controllability around stationary states of \eqref{systNLF} which can change of sign if $m \geq n-4$ (for technical reasons maybe, see \Cref{OpenProb}).

\subsection{Proof of the existence of invariant quantities in the system}\label{ProofPropinvquant}
The goal of this section is to prove \Cref{PropInvQuant}.
\begin{proof}
We introduce the notation $R :=\prod\limits_{k=1}^n u_k^{\alpha_k} - \prod\limits_{k=1}^n u_k^{\beta_k}$ and we take $m+1 \leq i \leq n$. By using the fact that $u_i \in W_T$ and from \cite[Lemma 3]{FCNavierCo}, we obtain that the mapping $t \mapsto \int_{\Omega} u_i(t,x) dx$ is absolutely continuous and for a.e. $0 \leq t \leq T$,
\begin{equation}
\label{dtuimass1}
\frac{d}{dt} \int_{\Omega} u_i(t,x) dx = \left(\partial_t u_i(t,.), 1\right)_{(H^1(\Omega))',H^1(\Omega)}.
\end{equation}
Then, by using that $((u_i)_{1 \leq i \leq n},(h_i)_{1 \leq i \leq m})$ is a trajectory of \eqref{systNLC} and by taking $w=1$ in \eqref{formvarNLC}, we find that for a.e. $0 \leq t \leq T$,
\begin{align}
\label{dtuimass2}\left(\partial_t u_i(t,.), 1\right)_{(H^1(\Omega))',H^1(\Omega)} &= d_i(\nabla u_i(t,.), \nabla 1)_{L^2(\Omega),L^2(\Omega)} + \int_{\Omega}(\beta_i-\alpha_i)R\\
&= \int_{\Omega}(\beta_i-\alpha_i)R. \notag
\end{align}
Then, by using \eqref{dtuimass1} and \eqref{dtuimass2}, we get for a.e. $0 \leq t \leq T$,
\begin{equation}
\label{dtuimass3}
\frac{d}{dt} \int_{\Omega} \frac{u_i(t,.)}{\beta_i-\alpha_i} =  \int_{\Omega}R.
\end{equation}
Now, let $m+1 \leq k \neq l \leq n$. By \eqref{dtuimass3} for $i=k$ and  \eqref{dtuimass3} for $i=l$ , we deduce that for a.e. $0 \leq t \leq T$,
\begin{equation}
\label{dtuimass4}
\frac{d}{dt} \int_{\Omega}\left(\frac{u_k(t,.)}{\beta_k-\alpha_k} - \frac{u_l(t,.)}{\beta_l-\alpha_l}\right) = 0.
\end{equation}
Therefore, from \eqref{dtuimass4}, we have for every $t \in [0,T]$,
\begin{equation*}
\frac{1}{|\Omega|} \int_{\Omega} \left(\frac{u_k(t,x)}{\beta_k-\alpha_k} - \frac{u_l(t,x)}{\beta_l-\alpha_l}\right)dx=\frac{u_k^{*}}{\beta_k-\alpha_k} - \frac{u_l^{*}}{\beta_l-\alpha_l}.
\label{ci1bis}
\end{equation*}
\indent If we assume that $d:=d_k=d_l$, then the equation satisfied by $v:= (\beta_l-\alpha_l) u_k - (\beta_k-\alpha_k) u_l$ is
\begin{equation}
\label{systvBackwardUni}
\left\{
\begin{array}{l l}
\partial_t v -  d \Delta v= 0&\mathrm{in}\ (0,T)\times\Omega,\\
\frac{\partial v}{\partial \nu} = 0 &\mathrm{on}\ (0,T)\times\partial\Omega,\\
v(T,.)= (\beta_l - \alpha_l)u_k^{*} - (\beta_k-\alpha_k)u_l^{*}&\mathrm{in}\  \Omega.
\end{array}
\right.
\end{equation}
The backward uniqueness of the heat equation (see for instance \cite[Théorème II.1]{BT}) applied to \eqref{systvBackwardUni} leads to 
\begin{equation*}
\forall t \in [0,T],\ (\beta_l - \alpha_l)u_{k}(t,.) - (\beta_k-\alpha_k)u_{l}(t,.)= (\beta_l - \alpha_l)u_k^{*} - (\beta_k-\alpha_k)u_l^{*}.
\end{equation*}
This yields \eqref{lienciedt}.
\end{proof}

\subsection{Proofs concerning the change of variables}

\subsubsection{Proof of the equivalence of the two systems}\label{Eq2systems}
In this section, we prove \Cref{propchangevarZ}. It is based on the following algebraic lemma.
\begin{lem}
\label{LemmaCombi}
Let $s$ be an integer such that $s \geq 2$. Let $(a_1, \dots, a_s) \in \C^{s}$ be such that $a_i \neq a_j$ for $i \neq j$. Then, we have
\begin{equation}
\sum\limits_{i=1}^{s} \prod \limits_{\substack{j=1 \\j \neq i}}^{s} \frac{1}{a_i-a_j} = 0.
\label{lemmacombieq}
\end{equation}
\end{lem}
\begin{proof}
Let $\C(X)$ be the field of fractional functions with coefficients in $\C$ and $F \in \C(X)$ be defined by
\begin{align}
\label{defFrational}
F(X) := \left(\sum\limits_{i=1}^{s-1}\left( \prod \limits_{\substack{j=1 \\j \neq i}}^{s-1}\frac{1}{a_i-a_j}  \right) \frac{1}{a_i - X} \right)+ \prod \limits_{j=1 }^{s-1} \frac{1}{X-a_j}.
\end{align}
The partial fractional decomposition of $F$ is the following one:
$$ F(X) = \sum\limits_{i=1}^{s-1} \frac{b_i}{X-a_{i}},\qquad\text{where}\ b_i \in \C.$$
For $1 \leq i \leq s-1$, we compute each $b_i$ by multiplying \eqref{defFrational} by $(X-a_i)$ and by evaluating $X=a_i$:
$$b_i = - \prod \limits_{\substack{j=1 \\j \neq i}}^{s-1}\frac{1}{a_i-a_j}+ \prod \limits_{\substack{j=1 \\j\neq i}}^{s-1}\frac{1}{a_i-a_j}=0.$$
We deduce that $F=0$. By remarking that $$F(a_s) = \sum\limits_{i=1}^{s} \prod \limits_{\substack{j=1\\ j \neq i}}^{s} \frac{1}{a_i-a_j} = 0,$$
we conclude the proof of \eqref{lemmacombieq} 
\end{proof}
The following result is an easy consequence of \Cref{LemmaCombi}.
\begin{cor}
\label{CorIdentity}
For every $m+2 \leq k \leq n$, we have
\begin{equation}
\label{identityLem}
\sum\limits_{l=m+1}^{k} P_{kl}(\beta_l-\alpha_l)=0.
\end{equation}
\end{cor}
\begin{proof}
By \eqref{DefcoeffP}, we have by taking $s=k-m$ and $a_i = d_{i+m}$ for $1 \leq i \leq k-m$ in \Cref{LemmaCombi}
$$ \sum\limits_{l=m+1}^{k} P_{kl}(\beta_l-\alpha_l) = \sum\limits_{i=1}^{k-m} P_{k,i+m}(\beta_{i+m}-\alpha_{i+m}) = \sum\limits_{i=1}^{k-m} \prod \limits_{\substack{j=1 \\j\neq i}}^{i-m} \frac{1}{d_{i+m}-d_{j+m}} = 0.$$
This ends the proof of \Cref{CorIdentity}.
\end{proof}
Now, we turn to the proof of \Cref{propchangevarZ}.
\begin{proof}
\indent We introduce the following notation:
$ R :=\prod\limits_{k=1}^n u_k^{\alpha_k} - \prod\limits_{k=1}^n u_k^{\beta_k}.$\\
\indent We assume that $(U,H^J)$ is a trajectory of \eqref{systNLC}. The equations $1 \leq i \leq m+1$ of \eqref{systNLCZ1} are clearly satisfied. Let $ m+2 \leq i \leq n$. We have:
\begin{align}
\notag \partial_t z_i - d_i \Delta z_{i} &=\partial_t \left(\sum_{j=m+1}^{i} P_{ij}u_j\right) - d_i \Delta  \left(\sum_{j=m+1}^{i} P_{ij}u_j\right)\\
\notag & = \sum_{j=m+1}^{i} P_{ij} (\partial_t u_j - d_j \Delta u_j + (d_j -d_i)\Delta u_j)\\
\label{PreuveEquiCalc}&=  \sum_{j=m+1}^{i} \left(P_{ij}((\beta_j - \alpha_j) R) + P_{ij}\underbrace{(d_j-d_i)}_{0\ \text{if}\ j=i}\Delta u_j\right)\\
\notag & = R \underbrace{\sum_{j=m+1}^{i} P_{ij}(\beta_j - \alpha_j)}_{0\ \text{by}\ \text{\Cref{CorIdentity}}} + \sum_{j=m+1}^{i-1} \underbrace{P_{ij}(d_j-d_i)}_{P_{i-1,j}\ \text{by}\ \eqref{DefcoeffP}} \Delta u_j\\
\notag & = \Delta z_{i-1}.
\end{align}
This ends the proof of “$\Rightarrow$”.\\
\indent We assume that $(Z,H^J)$ satisfies \eqref{systNLCZ1}. Then, the equations 
\begin{equation}
\label{preuveui}
\partial_t u_i - d_i \Delta u_i = (\beta_i - \alpha_i)R,
\end{equation}
are clearly satisfied for $1 \leq i \leq m+1$. We prove \eqref{preuveui} by strong induction on $i \in \{m+2,\dots,n\}$. By using \eqref{PreuveEquiCalc} for $i=m+2$ and \eqref{preuveui} for $i=m+1$, we obtain 
\begin{align*}
&\sum_{j=m+1}^{m+2}P_{m+2,j}(\partial_t u_{j} - d_{j} \Delta u_{j})=0 \\
&\Leftrightarrow  P_{m+2,m+2} (\partial_t u_{m+2} - d_{m+2} \Delta u_{m+2})= -R P_{m+2,m+1} (\beta_{m+1}-\alpha_{m+1}).
\end{align*}
 This leads to \eqref{preuveui} for $i=m+2$ by using $P_{m+2,m+1}/P_{m+2,m+2} = -(\beta_{m+2}-\alpha_{m+2})/(\beta_{m+1}-\alpha_{m+1})$ by \eqref{DefcoeffP}. For $i>m+2$, by induction, we have
$ P_{ii} (\partial_t u_{i} - d_{i} \Delta u_{i}) + \sum\limits_{j=m+1}^{i-1} P_{ij} (\beta_j-\alpha_j)R = 0$ by \eqref{PreuveEquiCalc}. Then, from \Cref{CorIdentity}, we have $\sum\limits_{j=m+1}^{i-1} P_{ij} (\beta_j - \alpha_j) = - P_{ii}(\beta_i-\alpha_i)$. This yields \eqref{preuveui} and ends the proof of “$\Leftarrow$”.\\
\indent This concludes the proof of \Cref{propchangevarZ}.
\end{proof}

\subsubsection{Proof of the equivalence concerning the mass condition}\label{Proofinvquantzu}
In this section, we prove the equivalence \eqref{InvquantZU} which leads to the equivalence between \Cref{LocContrLinfty} and \Cref{mainresult1}.
\begin{proof}
\indent Assume that $Z_0 \in L_{inv}^{\infty}$. Then, we have 
\begin{equation}
\label{EcritureInvquantZ}
\forall m+2 \leq i \leq n,\ \int_{\Omega} \sum_{k=m+1}^{i} P_{ik} (u_{k,0}(x) - u_k^{*}) dx = 0.
\end{equation}
We prove \eqref{ci1Eq} by strong induction on $k \geq m+2$. The case $k=m+2$ comes from \eqref{EcritureInvquantZ} for $i=m+2$ and $P_{m+2,m+1}/P_{m+2,m+2} = -(\beta_{m+2}-\alpha_{m+2})/(\beta_{m+1}-\alpha_{m+1})$ by \eqref{DefcoeffP}. For $i>m+2$ in \eqref{EcritureInvquantZ}, by induction, we have
$$ \int_{\Omega} \left\{P_{ii} (u_{i,0}(x) - u_i^{*}) + \sum\limits_{k=m+1}^{i-1} P_{ik} \frac{(\beta_k-\alpha_k)(u_{m+1,0}(x)-u_{m+1}^{*})}{\beta_{m+1}-\alpha_{m+1}}\right\}dx = 0.$$
Then, from \Cref{CorIdentity}, we have $\sum\limits_{k=m+1}^{i-1} P_{ik} (\beta_k - \alpha_k) = - P_{ii}(\beta_i-\alpha_i)$. This yields \eqref{ci1Eq} for $k=i$.\\
\indent Assume \eqref{ci1Eq} holds. From \Cref{CorIdentity}, we have that for every $m+2 \leq i \leq n$, 
$$ \int_{\Omega} \sum_{k=m+1}^{i} P_{ik} (u_{k,0}(.) - u_k^{*})  = \int_{\Omega} \sum_{k=m+1}^{i} P_{ik} \frac{\beta_k-\alpha_k}{\beta_{m+1}-\alpha_{m+1}} (u_{m+1,0}(.) - u_{m+1}^{*}) = 0.$$
This ends the proof of \eqref{InvquantZU}.
\end{proof}

\subsection{Proof of an observability estimate for linear finite dimensional systems}\label{ProofObsDimfinie}

The goal of this section is to give a self-contained proof of \Cref{lemmaCobs}. By the Hilbert Uniqueness Method (see \cite[Theorem 2.44]{C}), it suffices to show the following null-controllability result for finite dimensional systems.
\begin{prop}
There exist $C >0$, $p_1, p_2 \in \N$ such that for every $\tau \in (0,1)$, $\lambda \geq \lambda_1 $ with $\lambda_1$ the first positive eigenvalue of $(-\Delta, H_{Ne}^2(\Omega))$, $y_0 \in \R^n$, there exists a control $h \in L^2(0,\tau;\R^m)$ verifying
\begin{equation}
\label{EsticontrolFiniteDim}
\norme{h}_{L^2(0,T;\R^m)}^{2} \leq C \left(1 + \frac{1}{\tau^{p_1}} + \lambda^{p_2}\right) \norme{y_0}_{\R^n}^{2}
\end{equation}
such that the solution $y \in L^2(0,\tau;\R^n)$ of
\begin{equation}
\label{SystDimFiniGene}
\left\{
\begin{array}{ll}
y' = Ay + Bh,& \text{in}\ (0,\tau),\\
y(0)=y_0& \text{in}\ \R^n,
\end{array}
\right.
\end{equation}
where $A = -\lambda D_J + A_J$ (see \eqref{defNewD}, \eqref{defcouplA} and \eqref{CouplNonZeroA}) and $B = \begin{pmatrix} I_m\\
(0)\end{pmatrix}\in \mathcal{M}_{n,m}(\R)$,
satisfies $y(\tau)=0$.
\end{prop}
\begin{rmk}
We do not treat the case $\lambda_0=0$ with initial data $y_0 \in \R^{m+1} \times \{0\}^{n-m-1}$ because it is a simple adaptation of the following proof.
\end{rmk}
\begin{proof}
Let $\tau \in (0,1)$, $\lambda \geq \lambda_1 $, $y_0 \in \R^n$.\\
\indent \textbf{Step 1: Construction of the control $h$ by a Brunovsky approach.} We start by defining $\overline{y}$ to be the free solution of the system \eqref{SystDimFiniGene} (take $h = 0$). We have $\overline{y}(t) = e^{tA} y_0 = e^{t(-\lambda D_J + A_J)} y_0$. We easily have that for any $l \geq 0$, 
\begin{equation}
\label{estiZkexp}
\norme{\overline{y}^{(l)}}_{L^2(0,\tau;\R^n)} \leq C(1 + \lambda^{l-1/2})\norme{y_0}_{\R^n}.
\end{equation}
\indent We choose a cut-off function $\eta \in C^{\infty}([0,\tau]; \R)$ such that $\eta =1 $ on $[0,\tau/3]$ and $\eta = 0$ on $[2\tau/3, \tau]$ verifying:
\begin{equation}
\label{estimateeta}
\forall p \in \N,\ \forall t \in [0,\tau],\ |\eta^{(p)}(t)| \leq \frac{C_p}{\tau^p}.
\end{equation}
We start by choosing for every $i \in \{1, \dots, m-1,n\}$,
\begin{equation} \label{defZknZiCacadebis}y_{i}(t) := \eta(t) \overline{y_{i}}(t). \end{equation}
Then, by using the cascade form of \eqref{SystDimFiniGene}, we define by reverse induction on $i \in \{n-1, n-2, \dots, m+1\}$,
\begin{equation} \label{defZkn-1}y_i(t) := - \frac{1}{\lambda}\left(y_{i+1}'(t) + \lambda d_{i+1} y_{i+1}(t) \right).\end{equation}
Then, $y_{m}$ is defined by the equation number $(m+1)$ by
\begin{equation}
\label{defZkj}
y_{m}(t) := \frac{1}{a_{m+1,m}}\left(y_{m+1}'(t) + \lambda d_{m+1} y_{m+1}(t) -\sum\limits_{\substack{s=1\\s \neq m}}^{n} a_{m+1,s} y_{s}(t)\right).
\end{equation}
Finally, we set for the control
\begin{equation}
\label{defHJcascadebis}
h:= y' -A y.
\end{equation}
By \eqref{defZkj} and \eqref{defHJcascadebis}, $h$ is of the form $h=(h_1, \dots, h_m, 0, \dots, 0)$.\\
\indent \textbf{Step 2: Properties of the solution $y$ and estimate of the control $h$.} First, we remark that,
\begin{equation}
\label{solzbar}
\forall 1 \leq i \leq n,\ 
\left\{
\begin{array}{ll}
y_{i} = \overline{y_{i}}, & \text{in}\ [0,\tau/3],\\
y_{i} = 0, & \text{in}\ [2\tau/3,\tau].
\end{array}
\right.
\end{equation}
Indeed, the property \eqref{solzbar} is clear for $i \in \{1, \dots, m-1,n\}$ by definition \eqref{defZknZiCacadebis}. Then, we prove \eqref{solzbar} by reverse induction on $m \leq i \leq n$ by using \eqref{defZkn-1}, \eqref{defZkj} and the definition of $\overline{y}$, for instance, for $t \in [0,\tau/3]$:
$$y_{n-1}(t) = - \frac{1}{\lambda}\left(y_{n}'(t) + \lambda d_n y_{n}(t) \right) = - \frac{1}{\lambda}\left(\overline{y_{n}}'(t) + \lambda d_n \overline{y_{n}}(t)\right) =\overline{y_{n-1}}(t).$$
Now, we have by \eqref{defZknZiCacadebis}, \eqref{estimateeta} and \eqref{estiZkexp} that for every $i \in \{1, \dots, m-1\}$,
\begin{equation}
\label{estiZZ'2}
\sum\limits_{l=0}^1 \norme{y_i^{(l)}}_{L^2(0,\tau;\R^n)} \leq C \left(1+ \frac{1}{\tau^{1/2}} + \lambda^{1/2}\right) \norme{y_0}_{\R^n}.
\end{equation}
Then, we easily prove by reverse induction on $m \leq i \leq n$ by using \eqref{estiZkexp}, \eqref{estimateeta}, \eqref{defZknZiCacadebis}, \eqref{defZkn-1}, \eqref{defZkj} and \eqref{estiZZ'2}
\begin{align}
\label{estiZZ'1}
\sum\limits_{l=0}^{i+1-m} \norme{y_{i}^{(l)}}_{L^2(0,\tau;\R^n)} \leq C \left(1 + \frac{1}{\tau^{n-m+1/2}} + \lambda^{n-m+1/2}\right) \norme{y_{0}}_{L^2(0,\tau;\R^n)}.
\end{align}
\indent Hence, the control $h$ and the state $y$ satisfy \eqref{EsticontrolFiniteDim}, \eqref{SystDimFiniGene} with $p_1 = p_2 = 2(n-m+1/2)$ and $y(\tau)=0$.
\end{proof}

\subsection{Source term method in $L^r$ for $r \in \{2,+\infty\}$}
\label{AppSourceTermMethod}
We use the same notations as in the beginning of \Cref{SourcetermL^2}. The goal of this section is to prove \Cref{PropSourceTermL^2} and \Cref{ControlPropEstiLTTLinfty}.  We have the following result.
\begin{prop}\label{MethodLTTLinfty}
For every $S \in \mathcal{S}_r$ and $Z_0 \in L_{inv}^r$, there exists $H^J \in \mathcal{H}_r$, such that the solution $Z$ of \eqref{defsystZSourceS} satisfies $Z \in \mathcal{Z}_r$. Furthermore, there exists $C >0$, not depending on $S$ and $Z_0$, such that
\begin{equation}
\label{EstiLTTr}
\norme{Z/\rho_0}_{L^{\infty}([0,T];L^r(\Omega)^n)} + \norme{H^J}_{\mathcal{H}_r} \leq C_T \left( \norme{Z_0}_{L^{r}(\Omega)^n} + \norme{S}_{\mathcal{S}_r}\right),
\end{equation}
where $C_T = C e^{C/T}$.
\end{prop}
The proof is inspired by \cite[Proposition 2.6]{BM}.
\begin{proof}
\indent For $k \geq 0$, we define $T_k = T(1-q^{-k})$. On the one hand, let $a_0 = Z_0$ and, for $k \geq 0$, we define $a_{k+1} = Z_S(T_{k+1}^{-},.)$ where $Z_S$ is the solution to
\begin{equation}
\label{defsystZG}
\left\{
\begin{array}{l l}
 \partial_t Z_{S} - D_J \Delta Z_{S} = A_J Z_{S} + S&\mathrm{in}\ (0,T)\times\Omega,\\
\frac{\partial Z_{S}}{\partial \nu} = 0 &\mathrm{on}\ (0,T)\times\partial\Omega,\\
Z_{S}(T_k^{+},.)=0 &\mathrm{in}\  \Omega.
\end{array}
\right.
\end{equation}
From \Cref{wpl2linfty}, using the estimates \eqref{estl2faible} and \eqref{injclassique} for $r=2$  or \eqref{estl2faiblelinfty} and \eqref{injclassique} for $r=+\infty$, we have
\begin{equation}
\label{estilinfty_LTT}
\norme{a_{k+1}}_{L^{r}(\Omega)^n} \leq \norme{Z_{S}}_{L^{\infty}([T_k,T_{k+1}];L^{r}(\Omega)^n)} \leq C\norme{S}_{L^{r}((T_k,T_{k+1});L^{r}(\Omega)^n)}.
\end{equation}
On the other hand, for $k \geq 0$, we also consider the control systems
\begin{equation}
\label{defsystZH}
\left\{
\begin{array}{l l}
 \partial_t Z_{H} - D_J \Delta Z_{H} = A_J Z_{H} + H^{J} 1_{\omega}&\mathrm{in}\ (0,T)\times\Omega,\\
\frac{\partial Z_{H}}{\partial \nu} = 0 &\mathrm{on}\ (0,T)\times\partial\Omega,\\
Z_{H}(T_k^{+},.)=a_k &\mathrm{in}\  \Omega.
\end{array}
\right.
\end{equation}
Using \Cref{theolinear} for $r=2$ or \Cref{theolinearLinfty} for $r=+\infty$, we can define $H_k^{J} \in L^{r}((T_k,T_{k+1})\times\Omega)^m$ such that $Z_{H}(T_{k+1}^{-},.)=0$ and, thanks to the cost estimate \eqref{esticontrolTh} for $r=2$ or \eqref{esticontrolThLinfty} for $r=+\infty$ (recalling that $C_T \leq Me^{M/T}$),
\begin{equation}
\label{costlinftyTkTk+1} 
\norme{H_k^{J}}_{ L^{r}((T_k,T_{k+1})\times\Omega)^m} \leq M e^{\frac{M}{T_{k+1}-T_{k}}} \norme{a_k}_{L^2(\Omega)^n}.
\end{equation}
In particular, for $k=0$, we have
\begin{equation}
\label{costlinftyT0} 
\norme{H_0^{J}}_{ L^{r}((T_0,T_{1})\times\Omega)^m} \leq M e^{\frac{qM}{T(q-1)}} \norme{Z_0}_{L^2(\Omega)^n}.
\end{equation}
And, since $\rho_0$ is decreasing
\begin{equation}
\label{estiHT0}
\norme{H_0^{J}/\rho_0}_{L^{r}((T_0,T_{1})\times\Omega)^m} \leq  \rho_0^{-1}(T_1) M e^{\frac{qM}{T(q-1)}} \norme{Z_0}_{L^2(\Omega)^n}.
\end{equation}
For $k \geq 0$, since $\rho_{\mathcal{S}}$ is decreasing, combining \eqref{estilinfty_LTT} and \eqref{costlinftyTkTk+1} yields
\begin{equation}
\label{costH_k+1}
\norme{H_{k+1}^J}_{L^{r}((T_{k+1},T_{k+2})\times\Omega)^m} \leq C M e^{\frac{M}{T_{k+2}-T_{k+1}}} \rho_{\mathcal{S}}(T_k) \norme{S/\rho_{\mathcal{S}}}_{L^{r}((T_{k},T_{k+1})\times\Omega)^n}.
\end{equation}
In particular, by using $M e^{\frac{M}{T_{k+2}-T_{k+1}}} \rho_{\mathcal{S}}(T_k) = \rho_0(T_{k+2})$ (see \eqref{defrho0bis} and \eqref{defrhoG}), we have
\begin{align}
\label{costH_k+1bis}
\norme{H_{k+1}^J}_{L^{r}((T_{k+1},T_{k+2})\times\Omega)^m} & \leq C  \rho_0(T_{k+2}) \norme{S/\rho_{\mathcal{S}}}_{L^{r}((T_{k},T_{k+1})\times\Omega)^n}.
\end{align}
Then, from \eqref{costH_k+1bis}, by using the fact that $\rho_0$ is decreasing, 
\begin{equation}
\label{costH_k+1ter}
\norme{H_{k+1}^J/\rho_0}_{L^{r}((T_{k+1},T_{k+2})\times\Omega)^m}  \leq C  \norme{S/\rho_{\mathcal{S}}}_{L^{r}((T_{k},T_{k+1})\times\Omega)^n}.
\end{equation}
As in the original proof, we can paste the controls $H_{k}^J$ for $k \geq 0$ together by defining
\begin{equation}
H^{J} := \sum\limits_{k \geq 0} H_k^{J}.
\end{equation}
We have the estimate from \eqref{estiHT0} and \eqref{costH_k+1ter}
\begin{equation}
\norme{H^J}_{\mathcal{H}_r} \leq C \norme{S}_{\mathcal{S}_r} + C \rho_0^{-1}(T_1) M e^{\frac{qM}{T(q-1)}} \norme{Z_0}_{L^2(\Omega)^n}.
\end{equation}
The state $Z$ can also be reconstructed by concatenation of $Z_{S}$ + $Z_{H}$, which are continuous at each junction $T_k$ thanks to the construction. Then, we estimate the state. We use the energy estimate \eqref{estl2faible} for $r=2$ or \eqref{estl2faiblelinfty} for $r=+\infty$ from \Cref{wpl2linfty} on each time interval $(T_k, T_{k+1})$:
\begin{equation}
\norme{Z_S}_{L^{\infty}(T_k,T_{k+1};L^r(\Omega)^n)} \leq C \norme{S}_{L^{r}((T_k,T_{k+1})\times \Omega)^n},
\end{equation}
and
\begin{equation}
\norme{Z_H}_{L^{\infty}(T_k,T_{k+1};L^r(\Omega)^n)} \leq C\left( \norme{a_k}_{L^{r}(\Omega)^n} +  \norme{H_k^{J}}_{L^{r}((T_k,T_{k+1})\times \Omega)^m}\right).
\end{equation}
Proceeding similarly as for the estimate on the control, we obtain respectively
\begin{equation}
\norme{Z_S/\rho_0}_{L^{\infty}(T_k,T_{k+1};L^r(\Omega)^n)} \leq C M^{-1} \norme{S}_{\mathcal{S}_r},
\end{equation}
and
\begin{equation}
\norme{Z_H/\rho_0}_{L^{\infty}(T_k,T_{k+1};L^r(\Omega)^n)} \leq C M^{-1} \norme{S}_{\mathcal{S}_r} + C \rho_0^{-1}(T_1) M e^{\frac{qM}{T(q-1)}} \norme{Z_0}_{L^{\infty}(\Omega)^n}.
\end{equation}
Therefore, for an appropriate choice of constant $C >0$, $Z$ and $H^J$ satisfy \eqref{EstiLTTr}. This concludes the proof of \Cref{MethodLTTLinfty}.
\end{proof}

\subsection{Proof of a strong observability inequality}\label{proofStrongObs}
We take the same notations as in the beginning of \Cref{SourcetermL^2}. The goal of this section is to prove \Cref{CorStrongObs}.
\begin{proof}
We define $\mathcal{F}_1 : (Z_0, S) \in L_{inv}^2 \times \mathcal{S}_2 \mapsto Z(T,.) \in L_{inv}^2$, where $Z$ is the solution of \eqref{defsystZSourceS} with $H^J=0$ and $\mathcal{F}_2 : H^J \in \mathcal{H}_2  \mapsto Z(T,.) \in L_{inv}^2$ is the solution of \eqref{defsystZSourceS} with $(Z_0,S)=(0,0)$. It is easy to see that the null-controllability of \eqref{defsystZSourceS} is equivalent to $\text{Range}(\mathcal{F}_1) \subset \text{Range}(\mathcal{F}_2)$.\\
\indent From \cite[Lemma 2.48]{C}, we have that $\text{Range}(\mathcal{F}_1) \subset \text{Range}(\mathcal{F}_2)$ is equivalent to the observability inequality
\begin{equation}
\label{estiObsStronLem}\exists C_T > 0,\  \forall \varphi_T \in L_{inv}^2,\ \norme{\mathcal{F}_1^{*}(\varphi_T)}_{ L_{inv}^2 \times \mathcal{S}_2} \leq C_T \norme{\mathcal{F}_2^{*}(\varphi_T)}_{  \mathcal{H}_2}.
\end{equation}
Consequently, by using the null-controllability result for \eqref{defsystZSourceS}: \Cref{PropSourceTermL^2}, we have that \eqref{estiObsStronLem} holds true. Moreover, the constant $C_T$ in \eqref{estiObsStronLem} can be chosen such that $C_T \leq C e^{C/T}$ by using the cost estimate \eqref{EstiLTT} (see the proof of \cite[Theorem 2.44]{C} for more details between the constant of cost estimate and the constant of observability inequality).\\
\indent Duality arguments between $Z$, the solution of \eqref{defsystZSourceS}, and $\varphi$, the solution of \eqref{defsystZAdj}, lead to:
\begin{align*}\int_{\Omega} \mathcal{F}_1(Z_0,S)(x).\varphi_T(x)dx &= \int_{\Omega} Z_0(x).\varphi(0,x)dx + \int\int_{(0,T)\times\Omega} S.\varphi,\\
((Z_0,S),\mathcal{F}_1^{*}(\varphi_T))_{L^2(\Omega)^n \times \mathcal{S}_2}& =  \int_{\Omega} Z_0(x).\varphi(0,x)dx + \int\int_{(0,T)\times\Omega} S.\varphi \rho_{\mathcal{S}}^2 \rho_{\mathcal{S}}^{-2 },
\end{align*}
\begin{align*}\int_{\Omega} \mathcal{F}_2(H^J)(x).\varphi_T(x)dx &= \int\int_{(0,T)\times\omega} H^J.\varphi,\\
(H^J,\mathcal{F}_2^{*}(\varphi_T))_{\mathcal{H}_2}& =  \sum\limits_{i=1}^m \int\int_{(0,T)\times\Omega} h_i.\varphi_i \rho_{0}^2 1_{\omega} \rho_{0}^{-2 }.
\end{align*}
Consequently, by identification, we find
\begin{equation}
\label{IdentifAdj}
\mathcal{F}_1^{*}(\varphi_T) = (\varphi(0,.),\varphi \rho_{\mathcal{S}}^2) \in L^2(\Omega)^n \times \mathcal{S}_2,\ \quad\ \mathcal{F}_2^{*}(\varphi_T) = (\varphi_i \rho_{0}^2 1_{\omega} )_{1 \leq i \leq m}\in\mathcal{H}_2.
\end{equation}
Finally, by putting \eqref{IdentifAdj} in \eqref{estiObsStronLem}, we exactly obtain \eqref{ObsLTT} with $C_T = C e^{C/T}$. This ends the proof of \Cref{CorStrongObs}.
\end{proof}

\textbf{Acknowledgments.}
I would like to very much thank Karine Beauchard and Michel Pierre (Ecole Normale Supérieure de Rennes) for many fruitful, stimulating discussions, helpful advices. I am grateful to Joackim Bernier (Université de Rennes 1) for suggesting me the proof of \Cref{LemmaCombi}.

\small
\bibliographystyle{plain}
\small{\bibliography{bibliordnonlin}}

\end{document}